\documentclass[11pt]{article}
\usepackage{etex}
\reserveinserts{28}
\usepackage[all]{xy}

\usepackage{amsfonts}
\usepackage{amsthm}
\usepackage{enumerate}
\usepackage{graphicx}
\usepackage{mathrsfs}
\usepackage{bm}
\usepackage{cite}
\usepackage{amssymb,amsmath} 
\usepackage{makecell}
\usepackage{tikz}
\usepackage{pgf}
\usepackage{tikz}
\usetikzlibrary{patterns}
\usepackage{pgffor}
\usepackage{pgfcalendar}
\usepackage{pgfpages}
\usepackage{shuffle,yfonts}
\usepackage{mathtools}
\DeclareFontFamily{U}{shuffle}{}
\DeclareFontShape{U}{shuffle}{m}{n}{ <-8>shuffle7 <8->shuffle10}{}

\newcommand{\si}{\sigma}

\newcommand\Res{{\rm Res}}

\newcommand{\bfk}{{\boldsymbol{\sl{k}}}}

\newcommand{\bfx}{{\boldsymbol{\sl{x}}}}

 \allowdisplaybreaks
\usetikzlibrary{arrows,shapes,chains}
 \allowdisplaybreaks

\catcode`!=11
\let\!int\int \def\int{\displaystyle\!int}
\let\!lim\lim \def\lim{\displaystyle\!lim}
\let\!sum\sum \def\sum{\displaystyle\!sum}
\let\!sup\sup \def\sup{\displaystyle\!sup}
\let\!inf\inf \def\inf{\displaystyle\!inf}
\let\!cap\cap \def\cap{\displaystyle\!cap}
\let\!max\max \def\max{\displaystyle\!max}
\let\!min\min \def\min{\displaystyle\!min}
\let\!frac\frac \def\frac{\displaystyle\!frac}
\catcode`!=12

\let\oldsection\section
\renewcommand\section{\setcounter{equation}{0}\oldsection}

\allowdisplaybreaks

\DeclareMathOperator{\Li}{Li}

\DeclareMathOperator{\ti}{ti}
\DeclareMathOperator{\Ti}{Ti}
\DeclareMathOperator{\Si}{Si}
\DeclareMathOperator{\Ri}{Ri}

\def\N{\mathbb{N}}
\def\Z{\mathbb{Z}}
\def\Q{\mathbb{Q}}

\def\S{\widetilde{S}}

\theoremstyle{plain}
\newtheorem{thm}{Theorem}[section]
\newtheorem{lem}[thm]{Lemma}
\newtheorem{cor}[thm]{Corollary}
\newtheorem{con}[thm]{Conjecture}

\theoremstyle{definition}

\newtheorem{re}[thm]{Remark}
\newtheorem{exa}[thm]{Example}

\begin{document}
\title{\bf The Parity of Two Types of Cyclotomic Euler Sums via Contour Integrals}
\author{
{Ce Xu\thanks{Email: cexu2020@ahnu.edu.cn}}\\[1mm]
\small School of Mathematics and Statistics, Anhui Normal University,\\ \small Wuhu 241002, P.R. China
}
\date{}
\maketitle
\noindent{\bf Abstract.} In this paper, we employ methods of contour integration and residue calculus to investigate the parity of two classes of cyclotomic Euler-type sums. One class involves products of cyclotomic harmonic numbers, while the other involves products of cyclotomic odd harmonic numbers. We derive explicit formulas for the parity of linear and quadratic cases of these cyclotomic Euler-type sums and provide several illustrative examples. The results for the linear and quadratic cases ensure that we can provide explicit formulas for the parity of cyclotomic multiple $T$-values and cyclotomic multiple $S$-values up to depth three, both of which are even-odd variants of cyclotomic multiple zeta values. Furthermore, we present declarative theorems concerning the parity of these two types of cyclotomic Euler-type sums to arbitrary orders. Additionally, using contour integration techniques, we explore explicit linear and quadratic formulas for these cyclotomic Euler-type sums under more general conditions.

\medskip

\noindent{\bf Keywords}: Cyclotomic Euler-type sums; Contour integrals; Residue calculus; Parity; Cyclotomic multiple zeta values; Cyclotomic multiple $S/T$-values; Even-odd variants.
\medskip

\noindent{\bf AMS Subject Classifications (2020):} 11M32, 11M99.

\section{Introduction}
As is well known, the method of contour integration and residue computation in the complex plane is a highly effective approach for studying closed-form expressions of infinite series. Many classical problems involving infinite series can be proven using residue calculations. For example, the rational expansion of the classical inverse tangent function $\cot(z)$ can be demonstrated by considering the integral of the function $\cot(z) - 1/z$ along a square contour centered at the origin with side length $(n + 1/2)\pi\ (n\in\N)$, computing its residues, and taking the limit as $n$ approaches infinity. This yields:
\begin{align*}
\cot(z)=\frac1{z}+\sum_{n=1}^\infty \frac{2z}{z^2-n^2\pi^2}\quad (n\notin \Z\pi).
\end{align*}
Furthermore, according to \emph{Euler's formula} $e^{iz}=\cos(z)+i\sin(z)\ (i^2=-1,z\in \mathbb{C})$ and the definition of \emph{Bernoulli numbers} $\frac{x}{e^x-1}=\sum_{n=0}^{\infty}{\frac{B_n}{n!}x^n}$, the power series expansion of the inverse tangent function can be derived as follows (\cite{A2000}):
\begin{align*}
\frac{z}{2}\cot\frac{z}{2}=\frac{iz}{2}\frac{e^{iz/2}+e^{-iz/2}}{e^{iz/2}-e^{-iz/2}}=1-\sum_{n=1}^\infty (-1)^{n-1}\frac{B_{2n}}{(2n)!}z^{2n}.
\end{align*}
By combining the two expansions mentioned above, the rational fraction expansion of the inverse tangent function is first transformed into the form of a power series. Then, based on the uniqueness of power series, the closed-form expression for even zeta values (first discovered by Euler) is derived by equating the coefficients of any power of the variable $z$ in the two power series (see \cite[Thm. 1.2.4]{A2000}):
\begin{align}\label{Riemannevenzetaformula}
\zeta(2k)=\frac{(-1)^{k-1}B_{2k}(2\pi)^{2k}}{2(2k)!}\quad (k\in\N),
\end{align}
where $\zeta(s)$ denotes the \emph{Riemann zeta function} defined by
$$\zeta(s):=\sum_{n=1}^{\infty}{\frac{1}{n^s}}\quad(\Re(s)>1).$$
After discovering the formula for even zeta values, Euler, in his correspondence with Goldbach between 1742 and 1743, also investigated the relationship between the following double series
$S_{p;q}$ (now known as \emph{linear Euler sums} or \emph{double zeta star values}) and Riemann zeta values (see \cite{Euler1776} ). The series $S_{p;q}$ is defined as:
\begin{align}\label{defnlinearEulerSums}
S_{p;q}:=\sum_{n=1}^\infty \frac{H_n^{(p)}}{n^q}\quad (p,q\in \N\ \text{and}\ q>1),
\end{align}
where $H_n^{(p)}$ stands the {\emph{generalized harmonic number}} of order $p$ defined by
\[H_n^{(p)}:=\sum_{k=1}^n \frac{1}{k^p}.\]
Euler elaborated a method to show that the double series $S_{p;q}$ can be evaluated in terms of zeta values in the following cases: $p=1$, $p=q$, $(p,q)=(2,4),(4,2)$ and $p+q\leq 13$ odd. Euler even conjectured that all $S_{p;q}$ when $p+q$ is odd can be expressed in terms of a $\mathbb{Q}$-coefficient combination of zeta values, a result that was proven by Nielsen \cite{Nielsen1906} in 1906. Following Nielsen's work, numerous experts and scholars have provided different proofs of this conjecture. For instance, in 1998, Flajolet and Salvy \cite[Thm. 3.1]{Flajolet-Salvy} also demonstrated the conjecture using the method of contour integration and derived an explicit formula. Flajolet and Salvy not only considered the linear Euler sums $S_{p;q}$, but also systematically investigated the parity of \emph{generalized Euler sums} of the following form:
\begin{align}\label{classicalESdefn}
{S_{{p_1p_2\cdots p_k};q}} := \sum\limits_{n = 1}^\infty  {\frac{{H_n^{\left( {{p_1}} \right)}H_n^{\left( {{p_2}} \right)} \cdots H_n^{\left( {{p_k}} \right)}}}
{{{n^{q}}}}},
\end{align}
where $p_j\in \N\ (j=1,2,\ldots,k)$ and $q\geq 2$ with $1\leq p_1\leq p_2\leq \cdots \leq p_k$. The quantity $p_1+\cdots+p_k+q$ is called the ``weight" of the sum, and the quantity $k$ is called the ``degree (or order)". When $k\geq 2$, it is referred to as a \emph{nonlinear Euler sum}. They proved the parity theorem for generalized Euler sums by considering the residue computation of the contour integral $\oint_{(\infty)} f(s) ds$ of an integrand composed of trigonometric functions, the digamma function and its derivatives, and rational functions. Additionally, they provided general explicit formulae for the parity of linear and quadratic Euler sums, as well as some explicit formulas for cubic and quartic Euler sums. The parity theorem they established is stated as follows (see \cite[Thm. 5.3]{Flajolet-Salvy}): \emph{When the weight $p_1+\cdots+p_k+q$ and the depth $k$ are of the same parity, ${S_{{p_1p_2\cdots p_k};q}}$ is a $\Q$-linear combination of Euler sums of order at most $k-1$.} Here
\[\oint_{(\infty)} f(s) ds:=\lim_{R\rightarrow \infty}\oint_{|s|=R} r(s)\xi(s) ds =0,\]
where $|s|=R$ denote a circular contour with radius $R$. The $\xi(s)$ is a type of kernel function composed of trigonometric functions, the digamma function and its derivatives, while $r(s)$ is a rational function, typically taken to be $r(s)=1/s^q\ (q\in \N)$. The kernel function $\xi(s)$ is defined by the following two conditions: 1). \emph{$\xi(s)$ is meromorphic in the whole complex plane.} 2). \emph{$\xi(s)$ satisfies $\xi(s)=o(s)$ over an infinite collection of circles $\left| s \right| = {\rho _k}$ with ${\rho _k} \to \infty $.} In a series of papers \cite{Xu-Wang2020,Xu-Wang2022,Xu-Wang2023} by the authors of this article and Wang, also building upon the theory developed by Flajolet and Salvy, we investigated the parity of two other types of Euler sums, denoted as $T_{p_1,p_2,\ldots,p_k,q}^{\si_1,\si_2,\ldots,\si_k,\si}$ (called \emph{alternating Euler $T$-sums}) and $\S_{p_1,p_2,\ldots,p_k,q}^{\si_1,\si_2,\ldots,\si_k,\si}$ (called \emph{alternating Euler $\tilde{S}$-sums}). These two types of Euler sums are defined as:
\begin{align}
&T_{p_1,p_2,\ldots,p_k,q}^{\si_1,\si_2,\ldots,\si_k,\si}
    =\sum_{n=1}^\infty\si^{n-1}\frac{h_{n-1}^{(p_1)}(\si_1)h_{n-1}^{(p_2)}(\si_2)\cdots
        h_{n-1}^{(p_k)}(\si_k)}{(n-1/2)^q}\,,\label{Tsum.Unify}\\
&\S_{p_1,p_2,\ldots,p_k,q}^{\si_1,\si_2,\ldots,\si_k,\si}
    =\sum_{n=1}^\infty\si^{n-1}\frac{h_n^{(p_1)}(\si_1)h_n^{(p_2)}(\si_2)\cdots
        h_n^{(p_k)}(\si_k)}{n^q}\,,\label{Stsum.Unify}
\end{align}
where $(p_1,p_2,\ldots,p_k,q)\in\mathbb{N}^{k+1}$ and $(\si_1,\si_2,\ldots,\si_k,\si)\in\{\pm 1\}^{k+1}$ with $(q,\si)\neq (1,1)$. The $h_n^{(p)}(\si)$ denotes the \emph{(alternating) odd harmonic number} defined by
\begin{align*}
h_n^{(p)}(\si):=\sum_{k=1}^n \frac{\si^k}{(k-1/2)^p}.
\end{align*}
Research on classical Euler sums and other related sums has also been extensive. Some relevant findings in this area can be found in \cite{AC2017,AC2020,ChenEie2006,DilBoyadzhiev2015,MoZhou2025,SiXin2021,Sofo2018,WangLyu2018,ZhengYang2025,ZhengYang2025(1)}.

In two very recent papers \cite{Rui-Xu2025,Wang-Xu2025}, the authors of this paper, together with Rui and Wang, investigated the parity of cyclotomic Euler sums and cyclotomic Euler $T$-sums by considering contour integrals of integrands composed of generalized digamma functions, extended trigonometric functions, and rational functions. Here, the \emph{cyclotomic Euler sums} and \emph{cyclotomic Euler $T$-sums} are defined as:
\begin{align}
&S_{p_1,\ldots, p_k;q}(x_1,\ldots,x_k;x):=\sum_{n=1}^\infty \frac{\zeta_{n}(p_1;x_1)\zeta_{n}(p_2;x_2)\cdots \zeta_{n}(p_k;x_k)}{n^q}x^n,\\
&T_{p_1,\ldots, p_k;q}(x_1,\ldots,x_k;x):=\sum_{n=1}^\infty \frac{t_{n}(p_1;x_1)t_{n}(p_2;x_2)\cdots t_{n}(p_k;x_k)}{(n-1/2)^q}x^n,
\end{align}
where $p_1,\ldots,p_k,q\in \N$ and $x_1,\ldots,x_k,x$ are all roots of unity with $(q,x)\neq (1,1)$. Similar to classical Euler sums, we refer to the quantity $p_1+\cdots+p_k+q$ as the ``weight" of the sum, and the quantity $k$ as the ``order". If $k=0$, we denote $S_{\emptyset;q}(\emptyset;x):=\Li_q(x)$ and $T_{\emptyset;q}(\emptyset;x):=\ti_q(x)$. Here $\zeta_n(p;x)$ and $t_n(p;x)$ denote the finite sum of polylogarithm function and $t$-polylogarithm function, respectively, defined by
\begin{align}
\zeta_n(p;x):=\sum_{k=1}^n \frac{x^k}{k^p}\quad\text{and}\quad t_n(p;x):=\sum_{k=1}^n \frac{x^k}{(k-1/2)^p} \quad (p\in \N),
\end{align}
and the \emph{polylogarithm function} $\Li_p(x)$ and \emph{$t$-polylogarithm function} $\ti_p(x)$ are defined by
\begin{align}
&\Li_{p}(x):=\lim_{n\rightarrow \infty}\zeta_n(p;x)= \sum_{n=1}^\infty \frac{x^n}{n^p}\quad (|x|\leq 1,\ (p,x)\neq (1,1),\ p\in \N),\\
&\ti_{p}(x):=\lim_{n\rightarrow \infty}t_n(p;x)= \sum_{n=1}^\infty \frac{x^n}{(n-1/2)^p}\quad (|x|\leq 1,\ (p,x)\neq (1,1),\ p\in \N).
\end{align}
Moreover, due to the close connections between cyclotomic Euler sums, cyclotomic Euler $T$-sums, cyclotomic multiple zeta values, and cyclotomic multiple $t$-values, we further obtained some parity results and conjectures regarding cyclotomic multiple zeta values and cyclotomic multiple $t$-values in \cite{Rui-Xu2025,Wang-Xu2025}. For $\bfk=(k_1,\ldots,k_r)\in \N^r$ and $\bfx=(x_1,\ldots,x_r)$ (all $x_j$ are roots of unity) with $(k_r,x_r)\neq (1,1)$, the \emph{cyclotomic multiple zeta value} (\cite{YuanZh2014a,Zhao2007d,Zhao2010}) and \emph{cyclotomic multiple $t$-value} are defined by
\begin{align}
&\Li_{\bfk}(\bfx):=\sum_{0<n_1<\cdots<n_r} \frac{x_1^{n_1}\dotsm x_r^{n_r}}{n_1^{k_1}\dotsm n_r^{k_r}},\\
&\ti_{\bfk}(\bfx):=\sum_{0<n_1<\cdots<n_r} \frac{x_1^{n_1}\cdots x_r^{n_r}}{(n_1-1/2)^{k_1}\cdots (n_r-1/2)^{k_r}}.
\end{align}
If the terms $x_1,\ldots,x_r$ in the two expressions above are all $N$-th roots of unity, they are referred to as level $N$ cyclotomic multiple zeta values and level $N$ cyclotomic multiple $t$-values, respectively.
Cyclotomic multiple zeta values and cyclotomic multiple $t$-values are direct generalizations of the \emph{classical multiple zeta values}, which are defined as:
\begin{align*}
\zeta(\bfk)\equiv \zeta(k_1,\ldots,k_r):=\sum_{0<n_1<\cdots<n_r} \frac{1}{n_1^{k_1}\cdots n_r^{k_r}},
\end{align*}
where $k_1,\ldots,k_r$ are positive integers and $k_r\geq 2$. The $r$ and $k_1+\cdots+k_r$ are called the \emph{depth} and \emph{weight}, respectively. Multiple zeta values were independently introduced by Hoffman \cite{H1992} and Zagier \cite{DZ1994} in the 1990s. Over the past three decades, research on various properties of multiple zeta values has been extensive. For a detailed background and research results on multiple zeta values, refer to Zhao's monograph \cite{Z2016}. The classical multiple $t$-values were defined by Hoffman in 2018, see \cite{H2019}. For some recent research work on multiple $t$-values, refer to \cite{KomatsuLuca2025,LiSong2023,LiYan2023,LiYan2025,Zhao2015,Z2024} and references therein.
It is easily known from the stuffle relations that cyclotomic Euler sums can be expressed as $\Z$-linear combinations of cyclotomic multiple zeta values, while cyclotomic multiple $T$-values can be expressed as $\Z$-linear combinations of cyclotomic multiple $t$-values.

In this paper, we define the following two classes of Hurwitz-type cyclotomic Euler sums (respectively referred to as \emph{cyclotomic Euler $R$-sums} and \emph{cyclotomic Euler $\tilde{S}$-sums}):
\begin{align}
&R_{p_1,p_2,\ldots,p_k;q}(x_1,x_2,\ldots,x_k;x):=\sum_{n=0}^\infty \frac{\zeta_{n}(p_1;x_1)\zeta_{n}(p_2;x_2)\cdots \zeta_{n}(p_k;x_k)}{(n+1/2)^q}x^n,\label{defnS-ES1}\\
&\tilde{S}_{p_1,\ldots, p_k;q}(x_1,\ldots,x_k;x):=\sum_{n=1}^\infty \frac{t_{n}(p_1;x_1)t_{n}(p_2;x_2)\cdots t_{n}(p_k;x_k)}{n^q}x^n,\label{defnS-ES2}
\end{align}
where $p_1,\ldots,p_k,q\in \N$ and $x_1,\ldots,x_k,x$ are all roots of unity with $(q,x)\neq (1,1)$. In particular, if $k=0$ in \eqref{defnS-ES1} and \eqref{defnS-ES2}, then we let $R_{\emptyset;q}(\emptyset;x)=x^{-1}\ti_q(x)$ and $\tilde{S}_{\emptyset;q}(\emptyset;x):=\Li_q(x)$. The primary objective of this paper is to investigate the parity properties of cyclotomic Euler $R$-sums and cyclotomic Euler $\tilde{S}$-sums using the method of contour integration. Explicit formulas for the parity of linear and quadratic cyclotomic Euler $R$-sums and $S$-sums are derived. Based on the connections between linear and quadratic cyclotomic Euler $R$-sums, $\tilde{S}$-sums, and cyclotomic double and triple $S$-values and $T$-values, parity results for cyclotomic multiple $S$-values and $T$-values up to depth three are obtained. Finally, the method of contour integration is employed to provide several residue formulas for linear and quadratic cyclotomic Euler $R$-sums and $S$-sums in general forms. These cyclotomic multiple $S$-values and cyclotomic multiple $T$-values can be regarded as even-odd variants of cyclotomic multiple zeta values, defined respectively as:
\begin{align}
&\Si_{k_1,\ldots,k_r}(x_1,\ldots,x_r):=2^r\sum_{0<n_1<n_2<\cdots<n_r} \frac{x_1^{n_1}x_2^{n_2}\cdots x_r^{n_r}}{(2n_1)^{k_1}(2n_2-1)^{k_2}\cdots (2n_r-r+1)^{k_r}},\label{defn-CMSV}\\
&\Ti_{k_1,\ldots,k_r}(x_1,\ldots,x_r):=2^r\sum_{0<n_1<n_2<\cdots<n_r} \frac{x_1^{n_1}x_2^{n_2}\cdots x_r^{n_r}}{(2n_1-1)^{k_1}(2n_2-2)^{k_2}\cdots (2n_r-r)^{k_r}},\label{defn-CMTV}
\end{align}
where $k_1,\ldots,k_r\in \N$ and $x_1,\ldots,x_r$ are all roots of unity with $(k_r,x_r)\neq (1,1)$. The classical multiple $T$-values were defined by Kaneko-Tsumura in 2020, see \cite{KanekoTs2019}. The main results of this paper are the following two theorems (also see Theorems \ref{thm-parityc-CERS} and \ref{thm-parityc-CESS}):
\begin{thm}\label{thm-parityc-CERSmain} Let $x,x_1,\ldots,x_r$ be roots of unity, and $p_1,\ldots,p_r,q\geq 1$ with $(p_j,x_j)\neq (1,1)\ (j=1,2,\ldots,r)$ and $ (q,x)\neq (1,1)$. The
\begin{align*}
&xR_{p_1,\ldots,p_r;q}\Big(x_1,\ldots,x_r;x\Big)+(-1)^{p_1+\cdots+p_r+q+r}R_{p_1,\ldots,p_r;q}\Big(x_1^{-1},\ldots,x_r^{-1};x^{-1}\Big)
\end{align*}
can be expressed in terms of a rational combination of the products of cyclotomic Euler sums and cyclotomic Euler $R$-sums with order $<r$.
\end{thm}

\begin{thm}\label{thm-parityc-CE-SS} Let $x,x_1,\ldots,x_r$ be roots of unity, and $p_1,\ldots,p_r,q\geq 1$ with $(p_j,x_j)\neq (1,1)\ (j=1,2,\ldots,r)$ and $ (q,x)\neq (1,1)$. The
\begin{align*}
&(x_1\cdots x_r)^{-1}\tilde{S}_{p_1,\ldots,p_r;q}\Big(x_1,x_2,\ldots,x_r;x\Big)+(-1)^{p_1+\cdots+p_r+q+r}\tilde{S}_{p_1,\ldots,p_r;q}\Big(x_1^{-1},x_2^{-1},\ldots,x_r^{-1};x^{-1}\Big)
\end{align*}
reduces to a combination of cyclotomic Euler $\tilde{S}$-sums and cyclotomic Euler $R$-sums of lower orders.
\end{thm}

We may, in a similar manner, define another even-odd variant of cyclotomic multiple zeta values (called \emph{cyclotomic multiple $R$-values}) as follows:
\begin{align}
\Ri_{k_1,\ldots,k_r}(x_1,\ldots,x_r):=2^r\sum_{0<n_1<n_2<\cdots<n_r} \frac{x_1^{n_1}x_2^{n_2}\cdots x_r^{n_r}}{(2n_1)^{k_1}\cdots(2n_{r-1})^{k_2} (2n_r-1)^{k_r}},\label{defn-CMRV}
\end{align}
where $k_1,\ldots,k_r\in \N$ and $x_1,\ldots,x_k$ are all roots of unity with $(k_r,x_r)\neq (1,1)$. In particular, if $x_1,\ldots,x_r$ in \eqref{defn-CMRV} are all $N$-th roots of unity, they are referred to as \emph{level $N$ cyclotomic multiple $R$-values}. It is not difficult to observe, based on the stuffle relations, that a cyclotomic Euler $R$-sum of order $r$ can be expressed as a $\Z[i]$-linear combination of cyclotomic multiple $R$-values with depths not exceeding $r+1$. Based on our calculations and observations, we propose the following parity conjecture concerning cyclotomic multiple $R$-values:
\begin{con}\label{conjparitycmRv} Let $r>1$ and $x_1,\ldots,x_r$ be roots of unity, and $k_1,\ldots,k_r\geq 1$ with $(k_r,x_r)\neq (1,1)$. If $x_1,\ldots,x_r\in\{z\in \mathbb{C}: z^N=1\}$, then
\begin{align*}
&\Ri_{k_1,\ldots,k_r}(x_1,\ldots,x_r)=(-1)^{k_1+\cdots+k_r+r}x_r\Ri_{k_1,\ldots,k_r}\Big(x_1^{-1},\ldots,x_r^{-1}\Big) \quad(\text{mod  products}),
\end{align*}
where the ``mod products" means discarding all product terms of cyclotomic multiple zeta values and cyclotomic multiple $R $-values (which must appear) with depth less than $r$ and level less than or equal to $N$.
\end{con}

\section{Preliminaries}
This section mainly presents several important lemmas, which will play a crucial role in the subsequent consideration of residue computation for contour integrals.
\begin{lem}\emph{(cf.\ \cite{Flajolet-Salvy})}\label{lem-redisue-thm}
Let $\xi(s)$ be a kernel function and let $r(s)$ be a rational function which is $O(s^{-2})$ at infinity. Then
\begin{align}\label{residue-}
\sum\limits_{\alpha  \in O} {{\mathop{\rm Res}}{{\left( {r(s)\xi(s)},\alpha  \right)}}}  + \sum\limits_{\beta  \in S}  {{\mathop{\rm Res}}{{\left( {r(s)\xi(s)}, \beta  \right)}}}  = 0,
\end{align}
where $S$ is the set of poles of $r(s)$ and $O$ is the set of poles of $\xi(s)$ that are not poles $r(s)$. Here ${\mathop{\rm Re}\nolimits} s{\left( {r(s)},\alpha \right)} $ denotes the residue of $r(s)$ at $s= \alpha.$
\end{lem}
Moreover, in \cite{Wang-Xu2025} we also mentioned that the lemma above holds true for $r(s)\xi(s)=o(s^{-1})$ as well.
In Flajolet and Salvy \cite{Flajolet-Salvy}, the parity of the classical Euler sum \eqref{classicalESdefn} was studied by considering the contour integral of the integrand formed by the kernel function
\begin{align*}
\xi(s)=\frac{\pi\cot(\pi s)\psi^{(p_1-1)}(-s)\psi^{(p_2-1)}(-s)\cdots \psi^{(p_k-1)}(-s)}{(p_1-1)!(p_2-1)!\cdots(p_k-1)!}
\end{align*}
and the rational function $r(s) = 1/s^q\ (p_j,q\in \N,\ j=1,2,\ldots,k)$. Here $\psi(s)$ denotes the the \emph{digamma function} defined by
\begin{align}\label{defn-classical-psi-funtion}
\psi(s)=-\gamma-\frac1{s}+\sum_{k=1}^\infty \left(\frac1{k}-\frac{1}{s+k}\right),
\end{align}
where $s\in\mathbb{C}\setminus \N_0^-\ (\N_0^-:=\N^-\cup\{0\}=\{0,-1,-2,-3,\ldots\})$ and $\gamma$ denotes the \emph{Euler-Mascheroni constant}.

In \cite{Rui-Xu2025,Wang-Xu2025}, the authors investigated the parity of cyclotomic Euler sums and cyclotomic Euler $T$-sums by examining contour integrals
\begin{align}\label{contourCES-iN}
\oint\limits_{\left( \infty  \right)} \frac{\Phi(s;x)\phi^{(p_1-1)}(s;x_1)\cdots\phi^{(p_r-1)}(s;x_r)}{(p_1-1)!\cdots (p_r-1)!s^q}(-1)^{p_1+\cdots+p_r-r} ds=0
\end{align}
and
\begin{align}\label{contourCEtS-iN}
\oint\limits_{\left( \infty  \right)} \frac{\Phi(s;x)\phi^{(p_1-1)}(s+1/2;x_1)\cdots\phi^{(p_r-1)}(s+1/2;x_r)}{(p_1-1)!\cdots (p_r-1)!(s+1/2)^q}(-1)^{p_1+\cdots+p_r-r} ds=0,
\end{align}
of functions composed of a kernel function (constructed from generalized digamma functions $\phi(s;x)$ or $\phi(s+1/2;x)$  and extended trigonometric function $\Phi(s;x)$) and rational functions $r(s)=1/s^q$ or $1/(s+1/2)^q$. Notably, any polynomial forms comprising $\Phi$ and $\phi^{(j)}$ serve as kernel function. Here the \emph{extended trigonometric function} $\phi(s;x)$ and \emph{generalized digamma function} $\Phi(s;x)$ are defined by
\begin{align}
\phi(s;x):=\sum_{k=0}^\infty \frac{x^k}{k+s}\quad (s\notin\N_0^-),
\end{align}
where $x$ is an arbitrary complex number with $|x|\leq 1$ and $x\neq 1$, and
\[\Phi(s;x):=\phi(s;x)-\phi\Big(-s;x^{-1}\Big)-\frac1{s},\]
where to ensure the convergence of the series above, $x$ can only be any root of unity. Studying the parity of cyclotomic Euler sums and cyclotomic Euler $T$-sums involves examining the residue computations of the contour integrals in \eqref{contourCES-iN} and \eqref{contourCEtS-iN}. The most critical step is to provide the Laurent series expansions or power series expansions of the $\phi$ and $\Phi$ functions at integer points and half-integer points. In \cite{Rui-Xu2025,Wang-Xu2025}, the authors correspondingly provided these expansions (see the following lemmas):
\begin{lem}(\cite[Eqs. (2.4) and (2.5)]{Rui-Xu2025})\label{lem-rui-xu-one} For $p\in \N$, if $|s+n|<1\ (n\geq 0)$, then
\begin{align}
\frac{\phi^{(p-1)}(s;x)}{(p-1)!} (-1)^{p-1}&=x^n\sum_{k=0}^\infty \binom{k+p-1}{p-1} \left((-1)^k \Li_{k+p}(x)+(-1)^p\zeta_n\Big(k+p;x^{-1}\Big)\right)(s+n)^k \nonumber\\&\quad+\frac{x^n}{(s+n)^p}\qquad (|s+n|<1,\ n\geq 0)\label{Lexp-phi--n-diffp-1},\\
\frac{\phi^{(p-1)}(s;x)}{(p-1)!} (-1)^{p-1}&=x^{-n}\sum_{k=0}^\infty \binom{k+p-1}{p-1} (-1)^k  \left( \Li_{k+p}(x)-\zeta_{n-1}\Big(k+p;x\Big)\right)(s-n)^k \nonumber\\&\qquad\qquad\qquad\qquad (|s-n|<1,\ n\geq 1).\label{Lexp-phi-n-diffp-1}
\end{align}
\end{lem}
\begin{lem}(\cite[Eq. (2.6)]{Rui-Xu2025})\label{lem-rui-xu-two} For $n\in \Z$,
\begin{align}\label{LEPhi-function}
\Phi(s;x)=x^{-n} \left(\frac1{s-n}+\sum_{m=0}^\infty \Big((-1)^m\Li_{m+1}(x)-\Li_{m+1}\Big(x^{-1}\Big)\Big)(s-n)^m \right).
\end{align}
\end{lem}

\begin{lem}(\cite[Lem. 2.4]{Wang-Xu2025})\label{lem-extend-rui-xu-one} For $p\in \N$, if $|s+n|<1\ (n\geq 0)$, then
\begin{align}\label{Lexp-phi---n-diffp-1}
&\frac{\phi^{(p-1)}(s+1/2;x)}{(p-1)!} (-1)^{p-1}\nonumber\\
&=x^n\sum_{k=0}^\infty \binom{k+p-1}{p-1} \left((-1)^k \ti_{k+p}(x)x^{-1}+(-1)^p t_n\Big(k+p;x^{-1}\Big)\right)(s+n)^k
\end{align}
and if $|s-n|<1\ (n\geq 1)$
\begin{align}\label{Lexp-phi---n-diffp-2}
&\frac{\phi^{(p-1)}(s+1/2;x)}{(p-1)!} (-1)^{p-1}\nonumber\\
&=x^{-n-1}\sum_{k=0}^\infty \binom{k+p-1}{p-1} (-1)^k  \left( \ti_{k+p}(x)-t_{n}\Big(k+p;x\Big)\right)(s-n)^k.
\end{align}
\end{lem}

\begin{lem}(\cite[Lem 2.5]{Wang-Xu2025})\label{lem-extend-rui-xu-two} If $|s+n+1/2|<1\ (n\geq 0)$, then
\begin{align}\label{n-1/2-lemma-MtV}
\Phi(s;x)=x^n \sum_{m=0}^\infty \left((-1)^m \ti_{m+1}(x)-x \ti_{m+1}\Big(x^{-1}\Big) \right)(s+n+1/2)^m
\end{align}
and if $|s-1/2|<1$, then
\begin{align}\label{1/2-lemma-MtV}
\Phi(s;x)=x^{-1} \sum_{m=0}^\infty \left((-1)^m \ti_{m+1}(x)-x \ti_{m+1}\Big(x^{-1}\Big) \right)(s-1/2)^m.
\end{align}
\end{lem}
It should be noted that equation \eqref{1/2-lemma-MtV} implies that \eqref{n-1/2-lemma-MtV} also holds for $n = -1$. In \cite[Lem 2.5]{Wang-Xu2025}, only (2.1) is provided, while \eqref{1/2-lemma-MtV} is not explicitly stated; however, \eqref{1/2-lemma-MtV} can also be obtained through direct calculation.

The contour integrals considered in this paper will also utilize the lemmas mentioned above. Additionally, we will need to introduce the following lemma:
\begin{lem}\label{lem-extend-xuzhou-three} If $|s-1/2|<1\ (n\geq 0)$, then
\begin{align}
\frac{\phi^{(p-1)}(s;x)}{(p-1)!}(-1)^{p-1}=\sum_{k=0}^\infty (-1)^k\binom{k+p-1}{p-1}\ti_{k+p}(x)x^{-1}(s-1/2)^k.
\end{align}
\end{lem}
\begin{proof}
By an elementary calculation, we deduce
\begin{align*}
\phi(s;x)=\sum_{k=0}^\infty (-1)^k\ti_{k+1}(x)x^{-1}(s-1/2)^k.
\end{align*}
Taking the $(p-1)$th derivative with respect to $s$ on both sides of the above equation completes the proof of this lemma.
\end{proof}

In this paper, we primarily investigate the parity results of cyclotomic Euler $R$-sums in \eqref{defnS-ES1} and cyclotomic Euler $\tilde{S}$-sums in \eqref{defnS-ES2} by examining the following two types of contour integrals:
\begin{align}
&\oint\limits_{\left( \infty  \right)} \frac{\Phi(s;x)\phi^{(p_1-1)}(s;x_1)\cdots\phi^{(p_r-1)}(s;x_r)}{(p_1-1)!\cdots (p_r-1)!(s-1/2)^q}(-1)^{p_1+\cdots+p_r-r} ds=0,\label{contourintegralonepaper}\\
&\oint\limits_{\left( \infty  \right)} \frac{\Phi(s;x)\phi^{(p_1-1)}(s+1/2;x_1)\cdots\phi^{(p_r-1)}(s+1/2;x_r)}{(p_1-1)!\cdots (p_r-1)!s^q}(-1)^{p_1+\cdots+p_r-r} ds=0.\label{contourintegraltwopaper}
\end{align}

\section{Cyclotomic Euler $R$-Sums}
In this section, we will derive explicit formulas for the parity of linear and quadratic cyclotomic Euler $R$-sums by examining the residue computations of the contour integrals in \eqref{contourintegralonepaper}, and provide specific illustrative examples. Furthermore, by exploring the connections between linear/quadratic cyclotomic Euler $R$-sums and cyclotomic double $S$-values as well as cyclotomic triple $S$-values, we will obtain parity results for cyclotopic multiple $S$-values with depth $\leq 3$.

\begin{thm}\label{thm-linearCES-1} Let $x,y$ be roots of unity, and $p,q\geq 1$ with $(p,y), (q,xy)\neq (1,1)$. We have
\begin{align}
&(xy)^{-1}R_{p;q}\Big(y;(xy)^{-1}\Big)-(-1)^{p+q}R_{p;q}\Big(y^{-1};xy\Big)\nonumber\\
&=\Li_p(y)\ti_q\Big((xy)^{-1}\Big)+(-1)^q(xy)^{-1}\Li_p(y)\ti_q(xy)+(-1)^q\binom{p+q-1}{p}(xy)^{-1}\ti_{p+q}(xy)\nonumber\\
&\quad+(-1)^q (xy)^{-1} \sum_{m=0}^{p-1}\binom{p+q-m-2}{q-1}\left((-1)^m\Li_{m+1}(x)-\Li_{m+1}\Big(x^{-1}\Big)\right)\ti_{p+q-m-1}(xy)\nonumber\\
&\quad+(-1)^q y^{-1} \sum_{m=0}^{q-1}\binom{p+q-m-2}{p-1}\ti_{p+q-m-1}(y)\left((-1)^m\ti_{m+1}\Big(x^{-1}\Big)-x^{-1}\ti_{m+1}(x)\right).
\end{align}
\end{thm}
\begin{proof}
The proof of this theorem requires consideration of this type of contour integral:
\begin{align*}
\oint\limits_{\left( \infty  \right)} H_{p,q}(x,y;s)ds:= \oint\limits_{\left( \infty  \right)} \frac{\Phi(s;x)\phi^{(p-1)}(s;y)}{(p-1)!(s-1/2)^q} (-1)^{p-1}ds=0.
\end{align*}
Obviously, all positive integers are simple poles, all non-positive integers are poles of order $p+1$, and $s=1/2$ is a pole of order $q$. Applying Lemmas \ref{lem-rui-xu-one}, \ref{lem-rui-xu-two}, \ref{lem-extend-rui-xu-two} and \ref{lem-extend-xuzhou-three}, by direct calculations, we deduce the following residues
\begin{align*}
&\Res\left(H_{p,q}(x,y;s),n\right)=\frac{(xy)^{-n}}{(n-1/2)^q}\left(\Li_p(y)-\zeta_{n-1}(p;y)\right)\quad (n\in \N),\\
&\Res\left(H_{p,q}(x,y;s),-n\right)=\frac1{p!}\lim_{s\rightarrow -n} \frac{d^p}{ds^p}\left\{(s+n)^{p+1}H_{p,q}(x,y;s)\right\}\quad (n\in\N_0)\\
&=(-1)^q\binom{p+q-1}{p}\frac{(xy)^n}{(n+1/2)^{p+q}}+(-1)^q\frac{(xy)^n}{(n+1/2)^q}\left(\Li_p(y)+(-1)^p\zeta_n\Big(p;y^{-1}\Big)\right)\\
&\quad+(-1)^q\sum_{m=0}^{p-1} \binom{p+q-m-2}{q-1}\left((-1)^m\Li_{m+1}(x)-\Li_{m+1}\Big(x^{-1}\Big)\right)\frac{(xy)^n}{(n+1/2)^{p+q-m-1}}
\end{align*}
and
\begin{align*}
&\Res\left(H_{p,q}(x,y;s),1/2\right)\\
&=\sum_{m+k=q-1,\atop m,k\geq 0} (-1)^k\binom{k+p-1}{p-1}\ti_{k+p}(y)y^{-1}\left((-1)^m x^{-1}\ti_{m+1}(x)- \ti_{m+1}\Big(x^{-1}\Big)\right).
\end{align*}
Applying Lemma \ref{lem-redisue-thm} and the three residue values mentioned above, the theorem can be proven through a straightforward calculation.
\end{proof}
\begin{exa}Setting $(p,q)=(1,2)$ in Theorem \ref{thm-linearCES-1}, we have
	\begin{align*}
		&(xy)^{-1}R_{1;2}\Big(y;(xy)^{-1}\Big)+R_{1;2}\Big(y^{-1};xy\Big)\nonumber\\
		&=\Li_1(y)\ti_2\Big((xy)^{-1}\Big)+(xy)^{-1}\Li_1(y)\ti_2(xy)+2(xy)^{-1}\ti_{3}(xy)\nonumber\\
		&\quad+(xy)^{-1}\left(\Li_{1}(x)-\Li_{1}\Big(x^{-1}\Big)\right)\ti_{2}(xy)\nonumber\\
		&\quad+y^{-1}\ti_{2}(y)\left(\ti_{1}\Big(x^{-1}\Big)-x^{-1}\ti_{1}(x)\right)-y^{-1}\ti_{1}(y)\left(\ti_{2}\Big(x^{-1}\Big)+x^{-1}\ti_{2}(x)\right).
	\end{align*}
	Setting $(p,q)=(2,1)$ in Theorem \ref{thm-linearCES-1}, we have
	\begin{align*}
	&(xy)^{-1}R_{2;1}\Big(y;(xy)^{-1}\Big)+R_{2;1}\Big(y^{-1};xy\Big)\nonumber\\
	&=\Li_2(y)\ti_1\Big((xy)^{-1}\Big)-(xy)^{-1}\Li_2(y)\ti_1(xy)-(xy)^{-1}\ti_{3}(xy)\nonumber\\
	&\quad-(xy)^{-1}\left(\Li_{1}(x)-\Li_{1}\Big(x^{-1}\Big)\right)\ti_{2}(xy)+(xy)^{-1}\left(\Li_{2}(x)+\Li_{2}\Big(x^{-1}\Big)\right)\ti_{1}(xy)\nonumber\\
	&\quad-y^{-1}\ti_{2}(y)\left(\ti_{1}\Big(x^{-1}\Big)-x^{-1}\ti_{1}(x)\right).
	\end{align*}
\end{exa}

Next, we consider the parity results for quadratic cyclotomic Euler $R$-sums. To better present the results in the subsequent discussion, we adopt the following notation unless otherwise specified:
\begin{align}
\sum_{\sigma\in \{(1,2),(2,1)\}} f(x_{\sigma(1)},x_{\sigma(2)}):=f(x_1,x_2)+f(x_2,x_1).
\end{align}

\begin{thm}\label{thm-unlinearCES-1} Let $x,x_1,x_2$ be roots of unity, and $p_1,p_2,q\geq 1$ with $(p_1,x_1),(p_2,x_2), (q,xx_1x_2)\neq (1,1)$. We have
	\begin{align}
		&(xx_1x_2)^{-1}R_{p_1,p_2;q}\Big(x_1,x_2;(xx_1x_2)^{-1}\Big)+(-1)^{p_1+p_2+q}R_{p_1,p_2;q}\Big(x_1^{-1},x_2^{-1};xx_1x_2\Big)\nonumber\\
		&=-\Li_{p_1}(x_1)\Li_{p_2}(x_2)\ti_q\Big((xx_1x_2)^{-1}\Big)\nonumber\\
		&\quad+(xx_1x_2)^{-1}\sum_{\sigma\in\left\{(12),(21)\right\}}\Li_{p_{\sigma(1)}}\big(x_{\sigma(1)}\big)R_{p_{\sigma(2)};q}\Big(x_{\sigma(2)};(xx_1x_2)^{-1}\Big)\nonumber\\
		&\quad-(-1)^q(xx_1x_2)^{-1}\binom{q+p_1+p_2-1}{q-1}\ti_{q+p_1+p_2}(xx_1x_2)\nonumber\\
		&\quad-(-1)^q\sum_{\sigma\in\left\{(12),(21)\right\}}\sum_{k=0}^{p_{\sigma(1)}}\binom{k+p_{\sigma(2)}-1}{p_{\sigma(2)}-1}\binom{q+p_{\sigma(1)}-k-1}{q-1}\nonumber\\
		&\qquad\qquad\qquad\qquad\times\begin{Bmatrix}
		(xx_1x_2)^{-1}(-1)^k\Li_{k+p_{\sigma(2)}}(x_{\sigma(2)})\ti_{q+p_{\sigma(1)}-k}(xx_1x_2)\\
			+(-1)^{p_{\sigma(2)}}R_{k+p_{\sigma(2)};p_{\sigma(1)}+q-k}\Big(x_{\sigma(2)}^{-1};xx_1x_2\Big)
		\end{Bmatrix}\nonumber\\
		&\quad-(-1)^q(xx_1x_2)^{-1}\sum_{k=0}^{p_1+p_2-1}\binom{q+p_1+p_2-k-2}{q-1}\nonumber\\
		&\qquad\qquad\qquad\qquad\times\left((-1)^k\Li_{k+1}(x)-\Li_{k+1}\Big(x^{-1}\Big)\right)\ti_{q+p_1+p_2-k-1}(xx_1x_2)\nonumber\\
		&\quad-(-1)^q(xx_1x_2)^{-1}\Li_{p_1}(x_1)\Li_{p_2}(x_2)\ti_q(xx_1x_2)\nonumber\\
		&\quad-(-1)^q\sum_{\sigma\in\left\{(12),(21)\right\}}(-1)^{p_{\sigma(2)}}\Li_{p_{\sigma(1)}}\big(x_{\sigma(1)}\big)R_{p_{\sigma(2)};q}\Big(x_{\sigma(2)}^{-1};xx_1x_2\Big)\nonumber\\
		&\quad-(-1)^q\sum_{\sigma\in\left\{(12),(21)\right\}}\sum_{0\le k_1+k_2\le p_{\sigma(1)}-1}\binom{k_2+p_{\sigma(2)}-1}{p_{\sigma(2)}-1}\binom{q+p_{\sigma(1)}-k_1-k_2-2}{q-1}\nonumber\\
		&\qquad\qquad\qquad\qquad\times\left((-1)^{k_1}\Li_{k_1+1}(x)-\Li_{k_1+1}\Big(x^{-1}\Big)\right)\nonumber\\
		&\qquad\qquad\qquad\qquad\times\begin{Bmatrix}
			(xx_1x_2)^{-1}(-1)^{k_2}\Li_{k_2+p_{\sigma(2)}}(x_{\sigma(2)})\ti_{q+p_{\sigma(1)}-k_1-k_2-1}(xx_1x_2)\\
			+(-1)^{p_{\sigma(2)}}R_{k_2+p_{\sigma(2)};p_{\sigma(1)}+q-k_1-k_2-1}\Big(x_{\sigma(2)}^{-1};xx_1x_2\Big)
		\end{Bmatrix}\nonumber\\
		&\quad-\sum_{k_1+k_2+k_3=q-1,\atop k_1,k_2,k_3\ge0}\left((-1)^{k_1}\ti_{k_1+1}(x)-x\ti_{k_1+1}\Big(x^{-1}\Big)\right)(-1)^{k_2+k_3}\binom{k_2+p_1-1}{p_1-1}\binom{k_3+p_2-1}{p_2-1}\nonumber\\
		&\qquad\qquad\qquad\qquad\times(xx_1x_2)^{-1}\ti_{k_2+p_1}(x_1)\ti_{k_3+p_2}(x_2).
	\end{align}
\end{thm}
\begin{proof}
	The proof of this theorem requires consideration of this type of contour integral:
	\begin{align*}
		\oint\limits_{\left( \infty  \right)} H_{p_1,p_2,q}(x,x_1,x_2;s)ds:= \oint\limits_{\left( \infty  \right)} \frac{\Phi(s;x)\phi^{(p_1-1)}(s;x_1)\phi^{(p_2-1)}(s;x_2)}{(p_1-1)!(p_2-1)!(s-1/2)^q} (-1)^{p_1+p_2}ds=0.
	\end{align*}
	Obviously, all positive integers are simple poles, all non-positive integers are poles of order $p_1+p_2+1$, and $s=1/2$ is a pole of order $q$. Applying Lemmas \ref{lem-rui-xu-one}, \ref{lem-rui-xu-two}, \ref{lem-extend-rui-xu-two} and \ref{lem-extend-xuzhou-three}, by direct calculations, we deduce the following residues:
	\begin{align*}
		&\Res\left(H_{p_1,p_2,q}(x,x_1,x_2;s),n\right)\\
		&=\frac{(xx_1x_2)^{-n}}{(n-1/2)^q}\left(\Li_{p_1}(x_1)-\zeta_{n-1}(p_1;x_1)\right)\left(\Li_{p_2}(x_2)-\zeta_{n-1}(p_2;x_2)\right)\quad (n\in \N),\\
		&\Res\left(H_{p_1,p_2,q}(x,x_1,x_2;s),-n\right)\\
		&=\frac1{(p_1+p_2)!}\lim_{s\rightarrow -n} \frac{d^{p_1+p_2}}{ds^{p_1+p_2}}\left\{(s+n)^{p_1+p_2+1}H_{p_1,p_2,q}(x,x_1,x_2;s)\right\}\quad (n\in\N_0)\\
		&=(-1)^q\binom{q+p_1+p_2-1}{q-1}\frac{(xx_1x_2)^n}{(n+1/2)^{q+p_1+p_2}}\\
		&\quad+(-1)^q\sum_{\sigma\in\left\{(12),(21)\right\}}\sum_{k=0}^{p_{\sigma(1)}}\binom{k+p_{\sigma(2)}-1}{p_{\sigma(2)}-1}\binom{q+p_{\sigma(1)}-k-1}{q-1}\\
		&\qquad\qquad\times\left((-1)^k\Li_{k+p_{\sigma(2)}}\big(x_{\sigma(2)}\big)+(-1)^{p_{\sigma(2)}}\zeta_n\Big(k+p_{\sigma(2)};x_{\sigma(2)}^{-1}\Big)\right)\frac{(xx_1x_2)^n}{(n+1/2)^{p_{\sigma(1)}+q-k}}\\
		&\quad+(-1)^q\sum_{k=0}^{p_1+p_2-1}\binom{q+p_1+p_2-k-2}{q-1}\left((-1)^k\Li_{k+1}(x)-\Li_{k+1}\Big(x^{-1}\Big)\right)\frac{(xx_1x_2)^n}{(n+1/2)^{q+p_1+p_2-k-1}}\\
		&\quad+(-1)^q\left(\Li_{p_1}(x_1)+(-1)^{p_1}\zeta_n\Big(p_1;x_1^{-1}\Big)\right)\left(\Li_{p_2}(x_2)+(-1)^{p_2}\zeta_n\Big(p_2;x_2^{-1}\Big)\right)\frac{(xx_1x_2)^n}{(n+1/2)^{q}}\\
		&\quad+(-1)^q\sum_{\sigma\in\left\{(12),(21)\right\}}\sum_{0\le k_1+k_2\le p_{\sigma(1)}-1}\binom{k_2+p_{\sigma(2)}-1}{p_{\sigma(2)}-1}\binom{q+p_{\sigma(1)}-k_1-k_2-2}{q-1}\nonumber\\
		&\qquad\qquad\times\left((-1)^{k_1}\Li_{k_1+1}(x)-\Li_{k_1+1}\Big(x^{-1}\Big)\right)\nonumber\\
		&\qquad\qquad\times
		\left((-1)^{k_2}\Li_{k_2+p_{\sigma(2)}}(x_{\sigma(2)})+(-1)^{p_{\sigma(2)}}\zeta_n\Big(k_2+p_{\sigma(2)};x_{\sigma(2)}^{-1}\Big)\right)\frac{(xx_1x_2)^n}{(n+1/2)^{p_{\sigma(1)}+q-k_1-k_2-1}}
	\end{align*}
	and
	\begin{align*}
		&\Res\left(H_{p_1,p_2,q}(x,x_1,x_2;s),1/2\right)=\sum_{k_1+k_2+k_3=q-1,\atop k_1,k_2,k_3\geq 0} \left((-1)^{k_1}\ti_{k_1+1}(x)- x\ti_{k_1+1}\Big(x^{-1}\Big)\right)\\
		&\qquad\qquad\qquad\times(-1)^{k_2+k_3}\binom{k_2+p_1-1}{p_1-1}\binom{k_3+p_2-1}{p_2-1}\ti_{k_2+p_1}(x_1)\ti_{k_3+p_2}(x_2)(xx_1x_2)^{-1}.
	\end{align*}
	Applying Lemma \ref{lem-redisue-thm} and the three residue values mentioned above, the theorem can be proven through a straightforward calculation.
\end{proof}
\begin{exa}Setting $(p_1,p_2,q)=(1,1,2)$ in Theorem \ref{thm-unlinearCES-1}, we have
\begin{align*}
	&(xx_1x_2)^{-1}R_{1,1;2}\Big(x_1,x_2;(xx_1x_2)^{-1}\Big)+R_{1,1;2}\Big(x_1^{-1},x_2^{-1};xx_1x_2\Big)\nonumber\\
	&=-\Li_{1}(x_1)\Li_{1}(x_2)\ti_2\Big((xx_1x_2)^{-1}\Big)+(xx_1x_2)^{-1}\Li_{1}\big(x_{1}\big)R_{1;2}\Big(x_{2};(xx_1x_2)^{-1}\Big)\\
	&\quad+(xx_1x_2)^{-1}\Li_{1}\big(x_{2}\big)R_{1;2}\Big(x_{1};(xx_1x_2)^{-1}\Big)-3(xx_1x_2)^{-1}\ti_{4}(xx_1x_2)\nonumber\\
	&\quad-2(xx_1x_2)^{-1}\Li_{1}(x_{2})\ti_{3}(xx_1x_2)+2R_{1;3}\Big(x_{2}^{-1};xx_1x_2\Big)+(xx_1x_2)^{-1}\Li_{2}(x_{2})\ti_{2}(xx_1x_2)\\
	&\quad+R_{2;2}\Big(x_{2}^{-1};xx_1x_2\Big)-2(xx_1x_2)^{-1}\Li_{1}(x_{1})\ti_{3}(xx_1x_2)+2R_{1;3}\Big(x_{1}^{-1};xx_1x_2\Big)\\
	&\quad+(xx_1x_2)^{-1}\Li_{2}(x_{1})\ti_{2}(xx_1x_2)+R_{2;2}\Big(x_{1}^{-1};xx_1x_2\Big)\\
	&\quad-2(xx_1x_2)^{-1}\left(\Li_{1}(x)-\Li_{1}\Big(x^{-1}\Big)\right)\ti_{3}(xx_1x_2)+(xx_1x_2)^{-1}\left(\Li_{2}(x)+\Li_{2}\Big(x^{-1}\Big)\right)\ti_{2}(xx_1x_2)\nonumber\\
	&\quad-(xx_1x_2)^{-1}\Li_{1}(x_1)\Li_{1}(x_2)\ti_2(xx_1x_2)+\Li_{1}\big(x_{1}\big)R_{{1};2}\Big(x_{2}^{-1};xx_1x_2\Big)\\
	&\quad+\Li_{1}\big(x_{2}\big)R_{{1};2}\Big(x_{1}^{-1};xx_1x_2\Big)-\left(\Li_{1}(x)-\Li_{1}\Big(x^{-1}\Big)\right)(xx_1x_2)^{-1}\Li_{1}(x_{2})\ti_{2}(xx_1x_2)\\
	&\quad+\left(\Li_{1}(x)-\Li_{1}\Big(x^{-1}\Big)\right)R_{1;2}\Big(x_{2}^{-1};xx_1x_2\Big)-\left(\Li_{1}(x)-\Li_{1}\Big(x^{-1}\Big)\right)(xx_1x_2)^{-1}\Li_{1}(x_{1})\ti_{2}(xx_1x_2)\\
	&\quad+\left(\Li_{1}(x)-\Li_{1}\Big(x^{-1}\Big)\right)R_{1;2}\Big(x_{1}^{-1};xx_1x_2\Big)+\left(\ti_{2}(x)+x\ti_{2}\Big(x^{-1}\Big)\right)(xx_1x_2)^{-1}\ti_{1}(x_1)\ti_{1}(x_2)\\
	&\quad+\left(\ti_{1}(x)-x\ti_{1}\Big(x^{-1}\Big)\right)(xx_1x_2)^{-1}\left(\ti_{2}(x_1)\ti_{1}(x_2)+\ti_{1}(x_1)\ti_{2}(x_2)\right).
	\end{align*}
Setting $(p_1,p_2,q)=(2,2,2)$ in Theorem \ref{thm-unlinearCES-1}, we have
	\begin{align*}
	&(xx_1x_2)^{-1}R_{2,2;2}\Big(x_1,x_2;(xx_1x_2)^{-1}\Big)+R_{2,2;2}\Big(x_1^{-1},x_2^{-1};xx_1x_2\Big)\nonumber\\
	&=-\Li_{2}(x_1)\Li_{2}(x_2)\ti_2\Big((xx_1x_2)^{-1}\Big)+(xx_1x_2)^{-1}\Li_{2}\big(x_{1}\big)R_{2;2}\Big(x_{2};(xx_1x_2)^{-1}\Big)\\
	&\quad+(xx_1x_2)^{-1}\Li_{2}\big(x_{2}\big)R_{2;2}\Big(x_{1};(xx_1x_2)^{-1}\Big)-5(xx_1x_2)^{-1}\ti_{6}(xx_1x_2)\nonumber\\
	&\quad-3(xx_1x_2)^{-1}\Li_{2}(x_{2})\ti_{4}(xx_1x_2)-3R_{2;4}\Big(x_{2}^{-1};xx_1x_2\Big)+4(xx_1x_2)^{-1}\Li_{3}(x_{2})\ti_{3}(xx_1x_2)\\
	&\quad-4R_{3;3}\Big(x_{2}^{-1};xx_1x_2\Big)-3(xx_1x_2)^{-1}\Li_{4}(x_{2})\ti_{2}(xx_1x_2)-3R_{4;2}\Big(x_{2}^{-1};xx_1x_2\Big)\\
	&\quad-3(xx_1x_2)^{-1}\Li_{2}(x_{1})\ti_{4}(xx_1x_2)-3R_{2;4}\Big(x_{1}^{-1};xx_1x_2\Big)+4(xx_1x_2)^{-1}\Li_{3}(x_{1})\ti_{3}(xx_1x_2)\\
	&\quad-4R_{3;3}\Big(x_{1}^{-1};xx_1x_2\Big)-3(xx_1x_2)^{-1}\Li_{4}(x_{1})\ti_{2}(xx_1x_2)-3R_{4;2}\Big(x_{1}^{-1};xx_1x_2\Big)\\
	&\quad-4(xx_1x_2)^{-1}\left(\Li_{1}(x)-\Li_{1}\Big(x^{-1}\Big)\right)\ti_{5}(xx_1x_2)+3(xx_1x_2)^{-1}\left(\Li_{2}(x)+\Li_{2}\Big(x^{-1}\Big)\right)\ti_{4}(xx_1x_2)\\
	&\quad-2(xx_1x_2)^{-1}\left(\Li_{3}(x)-\Li_{3}\Big(x^{-1}\Big)\right)\ti_{3}(xx_1x_2)+(xx_1x_2)^{-1}\left(\Li_{4}(x)+\Li_{4}\Big(x^{-1}\Big)\right)\ti_{2}(xx_1x_2)\\
	&\quad-(xx_1x_2)^{-1}\Li_{2}(x_1)\Li_{2}(x_2)\ti_2(xx_1x_2)-\Li_{2}\big(x_{1}\big)R_{2;2}\Big(x_{2}^{-1};xx_1x_2\Big)-\Li_{2}\big(x_{2}\big)R_{2;2}\Big(x_{1}^{-1};xx_1x_2\Big)\nonumber\\
	&\quad-2\left(\Li_{1}(x)-\Li_{1}\Big(x^{-1}\Big)\right)
		\left((xx_1x_2)^{-1}\Li_{2}(x_{2})\ti_{3}(xx_1x_2)
		+R_{2;3}\Big(x_{2}^{-1};xx_1x_2\Big)\right)\\
	&\quad+2\left(\Li_{1}(x)-\Li_{1}\Big(x^{-1}\Big)\right)\left((xx_1x_2)^{-1}\Li_{3}(x_{2})\ti_{2}(xx_1x_2)-R_{3;2}\Big(x_{2}^{-1};xx_1x_2\Big)\right)\\
	&\quad+\left(\Li_{2}(x)+\Li_{2}\Big(x^{-1}\Big)\right)
	\left((xx_1x_2)^{-1}\Li_{2}(x_{2})\ti_{2}(xx_1x_2)
	+R_{2;2}\Big(x_{2}^{-1};xx_1x_2\Big)\right)\\
	&\quad-2\left(\Li_{1}(x)-\Li_{1}\Big(x^{-1}\Big)\right)
	\left((xx_1x_2)^{-1}\Li_{2}(x_{1})\ti_{3}(xx_1x_2)
	+R_{2;3}\Big(x_{1}^{-1};xx_1x_2\Big)\right)\\
	&\quad+2\left(\Li_{1}(x)-\Li_{1}\Big(x^{-1}\Big)\right)\left((xx_1x_2)^{-1}\Li_{3}(x_{1})\ti_{2}(xx_1x_2)-R_{3;2}\Big(x_{1}^{-1};xx_1x_2\Big)\right)\\
	&\quad+\left(\Li_{2}(x)+\Li_{2}\Big(x^{-1}\Big)\right)
	\left((xx_1x_2)^{-1}\Li_{2}(x_{1})\ti_{2}(xx_1x_2)
	+R_{2;2}\Big(x_{1}^{-1};xx_1x_2\Big)\right)\\
	&\quad+\left(\ti_{2}(x)+x\ti_{2}\Big(x^{-1}\Big)\right)(xx_1x_2)^{-1}\ti_{2}(x_1)\ti_{2}(x_2)\\
	&\quad+2\left(\ti_{1}(x)-x\ti_{1}\Big(x^{-1}\Big)\right)(xx_1x_2)^{-1}\left(\ti_{3}(x_1)\ti_{2}(x_2)+\ti_{2}(x_1)\ti_{3}(x_2)\right).
\end{align*}
\end{exa}

From the definitions of cyclotomic Euler $R$-sums and cyclotomic multiple $S$-values, it follows that
\begin{align*}
\Si_{p,q}(x,y)&=4\sum_{0<n_1<n_2} \frac{x^{n_1} y^{n_2}}{(2n_1)^p(2n_2-1)^q}=\frac1{2^{p+q-2}} \sum_{n=1}^\infty \frac{y^n}{(n-1/2)^q}\zeta_{n-1}(p;x)=\frac{y}{2^{p+q-2}}R_{p;q}(x;y)
\end{align*}
and
\begin{align*}
\Si_{p,q,r}(x,y,z)&=8\sum_{0<n_1<n_2<n_3} \frac{x^{n_1} y^{n_2}z^{n_3}}{(2n_1)^p(2n_2-1)^q(2n_3-2)^r}=\frac1{2^{p+q+r-3}}\sum_{0<l<n<m}\frac{x^ly^nz^m}{l^p(n-1/2)^q(m-1)^r}\\
&=\frac1{2^{p+q+r-3}}\sum_{n=2}^\infty \frac{y^n}{(n-1/2)^q}\sum_{l=1}^{n-1} \frac{x^l}{l^p}\sum_{m=n+1}^\infty \frac{z^m}{(m-1)^r}\\
&=\frac1{2^{p+q+r-3}}\sum_{n=1}^\infty \frac{y^{n+1}}{(n+1/2)^q}\sum_{l=1}^{n} \frac{x^l}{l^p}\sum_{m=n+1}^\infty \frac{z^{m+1}}{m^r}\\
&=\frac{yz}{2^{p+q+r-3}}\sum_{n=1}^\infty \frac{y^{n}}{(n+1/2)^q} \zeta_n(p;x)\left(\Li_r(z)-\zeta_n(r;z)\right)\\
&=\frac{yz}{2^{p+q+r-3}} \Li_r(z)R_{p;q}(x;y)-\frac{yz}{2^{p+q+r-3}}R_{p,r,;q}(x,z;y).
\end{align*}
Combining Theorems \ref{thm-linearCES-1} and \ref{thm-unlinearCES-1}, we can thus obtain parity results for cyclotomic double $S$-values and cyclotomic triple $S$-values.
\begin{cor} Let $x,y$ be roots of unity, and $p,q\geq 1$ with $(p,x), (q,y)\neq (1,1)$. Then
\begin{align*}
\Si_{p,q}(x,y)-(-1)^{p+q}y\Si_{p,q}\Big(x^{-1},y^{-1}\Big)
\end{align*}
can be expressed in terms of a rational combination of products of cyclotomic single $t$-values and cyclotomic single zeta values.
\end{cor}

\begin{cor} Let $x,y,z$ be roots of unity, and $p,q\geq 1$ with $(p,x), (q,y), (r,z)\neq (1,1)$. Then
\begin{align*}
\Si_{p,q,r}(x,y,z)+(-1)^{p+q+r}yz^2\Si_{p,q,r}\Big(x^{-1},y^{-1},z^{-1}\Big)
\end{align*}
can be expressed in terms of a rational combination of the products of cyclotomic multiple $S$-values and cyclotomic multiple zeta values with depth $\leq 2$.
\end{cor}
Utilizing Panzer's results \cite[Thm 1.3]{Panzer2017} on the parity of multiple polylogarithms, which can be stated as follows:
for all $r\in \N$ and $\bfk=(k_1,\ldots,k_r)\in \N^r$, the function
\begin{align*}
\Li_{\bfk}(z_1,z_2,\ldots,z_r)-(-1)^{k_1+\cdots+k_r+r}\Li_{\bfk}(1/z_1,1/z_2,\ldots,1/z_r)
\end{align*}
is of depth at most $r-1$, where $(z_1,\ldots,z_r)\in \mathbb{C}^r \setminus \bigcup_{1\leq i\leq j\leq r}\{(z_1,\ldots,z_r): z_iz_{i+1}\cdots z_j\in[0,+\infty)\}$. We can thus derive a general parity conclusion regarding cyclotomic multiple $S$-values:
\begin{thm}\label{thm-weakendedofCMSVparity}
Let $r>1$ and $x_1,\ldots,x_r$ be roots of unity, and $k_1,\ldots,k_r\geq 1$ with $(k_r,x_r)\neq (1,1)$. If $x_1,\ldots,x_r\in\{z\in \mathbb{C}: z^N=1\}$, then
\begin{align*}
&\Si_{k_1,\ldots,k_r}(x_1,\ldots,x_r)-(-1)^{k_1+\cdots+k_r+r}(x_2x_3^2\cdots x_r^{r-1})\Si_{k_1,\ldots,k_r}\Big(x_1^{-1},\ldots,x_r^{-1}\Big)
\end{align*}
can be expressed in terms of a $\Q$-linear combination of cyclotomic multiple zeta values of level $\leq 2N$.
\end{thm}
\begin{proof} By a direct calculation, one obtains
\begin{align*}
&\Si_{k_1,\ldots,k_r}(x_1,\ldots,x_r)-(-1)^{k_1+\cdots+k_r+r}(x_2\cdots x_r^r)\Si_{k_1,\ldots,k_r}\Big(x_1^{-1},\ldots,x_r^{-1}\Big)\\
&=(\sqrt{x_2})\cdots (\sqrt{x_r})^{r-1}\sum_{\si_1,\ldots,\si_r\in\{\pm 1\}}\si_2 \si_3^2 \cdots \si_r^{r-1} \\&\qquad\qquad\quad\times \left(\Li_{k_1,\ldots,k_r}(\si_1\sqrt{x_1},\ldots,\si_r\sqrt{x_r})-(-1)^{k_1+\cdots+k_r+r}\Li_{k_1,\ldots,k_r}\left(\frac1{\si_1\sqrt{x_1}},\ldots,\frac1{\si_r\sqrt{x_r}}\right)\right).
\end{align*}
Therefore, by applying Panzer's parity theorem, the proof of the theorem can be completed.
\end{proof}

Similarly, a weakened version of Conjecture \ref{conjparitycmRv} can be proven using Panzer's results.
\begin{thm}\label{thm-weakendedofCMRVparity}
Let $r>1$ and $x_1,\ldots,x_r$ be roots of unity, and $k_1,\ldots,k_r\geq 1$ with $(k_r,x_r)\neq (1,1)$. If $x_1,\ldots,x_r\in\{z\in \mathbb{C}: z^N=1\}$, then
\begin{align*}
&\Ri_{k_1,\ldots,k_r}(x_1,\ldots,x_r)-(-1)^{k_1+\cdots+k_r+r}x_r\Ri_{k_1,\ldots,k_r}\Big(x_1^{-1},\ldots,x_r^{-1}\Big)
\end{align*}
can be expressed in terms of a $\Q$-linear combination of cyclotomic multiple zeta values of level $\leq 2N$.
\end{thm}
\begin{proof} By an elementary calculation, one obtains
\begin{align*}
&\Ri_{k_1,\ldots,k_r}(x_1,\ldots,x_r)-(-1)^{k_1+\cdots+k_r+r}x_r\Ri_{k_1,\ldots,k_r}\Big(x_1^{-1},\ldots,x_r^{-1}\Big)\\
&=\sqrt{x_r}\sum_{\si_1,\ldots,\si_r\in\{\pm 1\}}\si_r\\ &\quad\times\left(\Li_{k_1,\ldots,k_r}(\si_1\sqrt{x_1},\ldots,\si_r\sqrt{x_r})-(-1)^{k_1+\cdots+k_r+r}\Li_{k_1,\ldots,k_r}\left(\frac1{\si_1\sqrt{x_1}},\ldots,\frac1{\si_r\sqrt{x_r}}\right)\right).
\end{align*}
Therefore, by applying Panzer's parity theorem, the proof of the theorem can be completed.
\end{proof}

It is worth noting that Panzer's parity results did not provide explicit formulas, while Hirose \cite{Hirose2025} and Umezawa \cite{Umezawa2025arxiv} have recently respectively provided explicit formulas for the parity of multiple zeta values and multiple polylogarithms.

\section{Cyclotomic Euler $\tilde{S}$-Sums}

In this section, we will derive explicit formulas for the parity of linear and quadratic cyclotomic Euler $\tilde{S}$-sums by examining the residue computations of the contour integrals in \eqref{contourintegraltwopaper}, and provide specific illustrative examples. Furthermore, by exploring the connections between linear/quadratic cyclotomic Euler $\tilde{S}$-sums and cyclotomic double $T$-values as well as cyclotomic triple $T$-values, we will obtain parity results for cyclotopic multiple $S$-values with depth $\leq 3$. It should be noted that in \cite{Wang-Xu2025}, the author of this paper and Wang also mentioned the definition of cyclotomic Euler $\tilde{S}$-sums but did not elaborate on it in detail.

\begin{thm}\label{thm-linearCES-2} Let $x,y$ be roots of unity, and $p,q\geq 1$ with $(p,y), (q,xy)\neq (1,1)$. We have
\begin{align}
&y^{-1}\tilde{S}_{p;q}\Big(y;(xy)^{-1}\Big)-(-1)^{p+q} \tilde{S}_{p;q}\Big(y^{-1};xy\Big)\nonumber\\
&=y^{-1}\Li_q\Big((xy)^{-1}\Big)\ti_p(y)+(-1)^q y^{-1}\Li_{q}(xy)\ti_p(y)+(-1)^q\binom{p+q-1}{p-1}y^{-1}\ti_{p+q}(y)\nonumber\\
&\quad+(-1)^q (xy)^{-1}\sum_{m=0}^{p-1} \binom{p+q-m-2}{q-1}\ti_{p+q-m-1}(xy)\nonumber\\&\qquad\qquad\qquad\qquad\qquad\times\left((-1)^m\ti_{m+1}(x)-x \ti_{m+1}\Big(x^{-1}\Big)\right)\nonumber\\
&\quad+(-1)^q y^{-1} \sum_{m=0}^{q-1} \binom{p+q-m-2}{p-1}\ti_{p+q-m-1}(y)\left((-1)^m \Li_{m+1}\Big(x^{-1}\Big)-\Li_{m+1}(x)\right).
\end{align}
\end{thm}

\begin{proof}
In the context of this paper, the proof of this theorem is based on residue computations of the following contour integral:
\begin{align*}
\oint\limits_{\left( \infty  \right)} G_{p,q}(x,y;s)ds:= \oint\limits_{\left( \infty  \right)} \frac{\Phi(s;x)\phi^{(p-1)}(s+1/2;y)}{(p-1)!s^q} (-1)^{p-1}ds=0.
\end{align*}
Obviously, $n (n\in\N),\ 0,$ and $-(n + 1/2)\ (n\in \N_0)$ are the simple poles, $(q+1)$th-order poles, and $p$th-order poles of the integrand $G_{p,q}(x,y;s)$, respectively. Using Lemma \ref{lem-rui-xu-one}-\ref{lem-extend-rui-xu-two}, the following residue values can be obtained through direct calculation:
\begin{align*}
&\Res\left(G_{p,q}(x,y;s),-n\right)=(-1)^q\frac{(xy)^n}{n^q}\left(\ti_p(y)y^{-1}+(-1)^pt_n\Big(p;y^{-1}\Big) \right) \quad (n\geq 1),\\
&\Res\left(G_{p,q}(x,y;s),n\right)=\frac{x^{-n}y^{-n-1}}{n^q}\left(\ti_p(y)-t_n(p;y)\right)\quad (n\geq 1),\\
&\Res\left(G_{p,q}(x,y;s),-n-1/2\right)=\frac1{(p-1)!} \lim_{s\rightarrow -n-1/2} \frac{d^{p-1}}{ds^{p-1}}\left((s+n+1/2)^pG_{p,q}(x,y;s)\right)\quad (n\geq 0)\\
&=(-1)^q\sum_{m=0}^{p-1} \binom{p+q-m-2}{q-1} \left((-1)^m \ti_{m+1}(x)-x \ti_{m+1}\Big(x^{-1}\Big) \right) \frac{(xy)^n}{(n+1/2)^{p+q-m-1}}
\end{align*}
and
\begin{align*}
&\Res\left(G_{p,q}(x,y;s),0\right)=\frac1{q!} \lim_{s\rightarrow 0} \frac{d^{q}}{ds^{q}}\left(s^{q+1}G_{p,q}(x,y;s)\right)\\
&=(-1)^q\binom{p+q-1}{p-1}y^{-1}\ti_{p+q}(y)\\&\quad+\sum_{m+k=q-1,\atop m,k\geq 0} \binom{k+p-1}{p-1}(-1)^k \ti_{k+p}(y)y^{-1}\left((-1)^m\Li_{m+1}(x)-\Li_{m+1}\Big(x^{-1}\Big)\right).
\end{align*}
Applying Lemma \ref{lem-redisue-thm} in conjunction with the aforementioned four residue values suffices to prove the theorem.
\end{proof}
\begin{exa}Setting $(p,q)=(1,2)$ in Theorem \ref{thm-linearCES-2}, we have
	\begin{align*}
		&y^{-1}\tilde{S}_{1;2}\Big(y;(xy)^{-1}\Big)+ \tilde{S}_{1;2}\Big(y^{-1};xy\Big)\nonumber\\
		&=y^{-1}\Li_2\Big((xy)^{-1}\Big)\ti_1(y)+ y^{-1}\Li_{2}(xy)\ti_1(y)+y^{-1}\ti_{3}(y)\nonumber\\
		&\quad+(xy)^{-1}\ti_{2}(xy)\left(\ti_{1}(x)-x \ti_{1}\Big(x^{-1}\Big)\right)\nonumber\\
		&\quad+y^{-1}\ti_{2}(y)\left(\Li_{1}\Big(x^{-1}\Big)-\Li_{1}(x)\right)-y^{-1}\ti_{1}(y)\left(\Li_{2}\Big(x^{-1}\Big)+\Li_{2}(x)\right).
	\end{align*}
	Setting $(p,q)=(2,1)$ in Theorem \ref{thm-linearCES-2}, we have
	\begin{align*}
		&y^{-1}\tilde{S}_{2;1}\Big(y;(xy)^{-1}\Big)+ \tilde{S}_{2;1}\Big(y^{-1};xy\Big)\nonumber\\
		&=y^{-1}\Li_1\Big((xy)^{-1}\Big)\ti_2(y)- y^{-1}\Li_{1}(xy)\ti_2(y)-2y^{-1}\ti_{3}(y)\nonumber\\
		&\quad-(xy)^{-1}\ti_{2}(xy)\left(\ti_{1}(x)-x \ti_{1}\Big(x^{-1}\Big)\right)+(xy)^{-1}\ti_{1}(xy)\left(\ti_{2}(x)+x \ti_{2}\Big(x^{-1}\Big)\right)\nonumber\\
		&\quad-y^{-1}\ti_{2}(y)\left(\Li_{1}\Big(x^{-1}\Big)-\Li_{1}(x)\right).
	\end{align*}
\end{exa}

\begin{thm}\label{thm-unlinearCES-2} Let $x,x_1,x_2$ be roots of unity, and $p_1,p_2,q\geq 1$ with $(p_1,x_1),(p_2,x_2),(q,xx_1x_2)\neq (1,1)$. We have
	\begin{align}
		&(x_1x_2)^{-1}\tilde{S}_{p_1,p_2;q}\Big(x_1,x_2;(xx_1x_2)^{-1}\Big)+(-1)^{p_1+p_2+q} \tilde{S}_{p_1,p_2;q}\Big(x_1^{-1},x_2^{-1};xx_1x_2\Big)\nonumber\\
		&=-(x_1x_2)^{-1}\ti_{p_1}(x_1)\ti_{p_2}(x_2)\Li_{q}\Big((xx_1x_2)^{-1}\Big)-(-1)^q(x_1x_2)^{-1}\ti_{p_1}(x_1)\ti_{p_2}(x_2)\Li_{q}(xx_1x_2)\nonumber\\
		&\quad+\sum_{\sigma\in\left\{(12),(21)\right\}}(x_1x_2)^{-1}\ti_{p_{\sigma(1)}}\big(x_{\sigma(1)}\big)\tilde{S}_{p_{\sigma(2)};q}\Big(x_{\sigma(2)};(xx_1x_2)^{-1}\Big)\nonumber\\
		&\quad-(-1)^q\sum_{\sigma\in\left\{(12),(21)\right\}}(-1)^{p_{\sigma(2)}}x_{\sigma(1)}^{-1}\ti_{p_{\sigma(1)}}\big(x_{\sigma(1)}\big)\tilde{S}_{p_{\sigma(2)};q}\Big(x_{\sigma(2)}^{-1};xx_1x_2\Big)\nonumber\\
		&\quad-(-1)^q(xx_1x_2)^{-1}\sum_{k=0}^{p_1+p_2-1} \binom{p_1+p_2+q-k-2}{q-1} \nonumber\\
		&\quad\quad\times\left((-1)^k \ti_{k+1}(x)-x \ti_{k+1}\Big(x^{-1}\Big) \right) \ti_{p_1+p_2+q-k-1}(xx_1x_2)\nonumber\\
		&\quad-(-1)^q\sum_{\sigma\in\left\{(12),(21)\right\}}\sum_{0\le k_1+k_2\le p_{\sigma(1)}-1}\binom{p_{\sigma(1)}+q-k_1-k_2-2}{q-1}\binom{k_2+p_{\sigma(2)}-1}{p_{\sigma(2)}-1}\nonumber\\
		&\quad\quad\times\left((-1)^{k_1}\ti_{k_1+1}(x)-x\ti_{k_1+1}\Big(x^{-1}\Big)\right)\nonumber\\
		&\quad\quad\times\begin{Bmatrix}
			(-1)^{k_2}(xx_1x_2)^{-1}\Li_{k_2+p_{\sigma(2)}}\big(x_{\sigma(2)}\big)\ti_{q+p_{\sigma(1)}-k_1-k_2-1}(xx_1x_2)\\+(-1)^{p_{\sigma(2)}}R_{k_2+p_{\sigma(2)};q+p_{\sigma(1)}-k_1-k_2-1}\Big(x_{\sigma(2)}^{-1};(xx_1x_2)\Big)
		\end{Bmatrix}\nonumber\\
		&\quad-\sum_{k_1+k_2=q,\atop k_1,k_2\geq 0}(-1)^q\binom{p_1+k_1-1}{p_1-1}\binom{p_2+k_2-1}{p_2-1}(x_1x_2)^{-1}\ti_{p_1+k_1}(x_1)\ti_{p_2+k_2}(x_2)\nonumber\\
		&\quad-\sum_{k_1+k_2+k_3=q-1,\atop k_1,k_2,k_3\geq 0}\left((-1)^{k_1}\Li_{k_1+1}(x)-\Li_{k_1+1}\Big(x^{-1}\Big)\right)\nonumber\\
		&\quad\quad\times(-1)^{k_2+k_3}\binom{k_2+p_1-1}{p_1-1}\binom{k_3+p_2-1}{p_2-1} \ti_{k_2+p_1}(x_1)\ti_{k_3+p_2}(x_2)(x_1x_2)^{-1}.
	\end{align}
\end{thm}

\begin{proof}
	In the context of this paper, the proof of this theorem is based on residue computations of the following contour integral:
	\begin{align*}
		\oint\limits_{\left( \infty  \right)} G_{p_1,p_2,q}(x,x_1,x_2;s)ds:= \oint\limits_{\left( \infty  \right)} \frac{\Phi(s;x)\phi^{(p_1-1)}(s+1/2;x_1)\phi^{(p_2-1)}(s+1/2;x_2)}{(p_1-1)!(p_2-1)!s^q} (-1)^{p_1+p_2}ds=0.
	\end{align*}
	Obviously, $n (n\in\N),\ 0,$ and $-(n + 1/2)\ (n\in \N_0)$ are the simple poles, $(q+1)$th-order poles, and $(p_1+p_2)$th-order poles of the integrand $G_{p_1,p_2,q}(x,x_1,x_2;s)$, respectively. Using Lemma \ref{lem-rui-xu-one}-\ref{lem-extend-rui-xu-two}, the following residue values can be obtained through direct calculation:
	\begin{align*}
		&\Res\left(G_{p_1,p_2,q}(x,x_1,x_2;s),-n\right)\quad (n\geq 1)\\
		&=(-1)^q\frac{(xx_1x_2)^n}{n^q}\left(\ti_{p_1}(x_1)x_1^{-1}+(-1)^{p_1}t_n\Big(p_1;x_1^{-1}\Big) \right)\left(\ti_{p_2}(x_2)x_2^{-1}+(-1)^{p_2}t_n\Big(p_2;x_2^{-1}\Big) \right) ,\\
		&\Res\left(G_{p_1,p_2,q}(x,x_1,x_2;s),n\right)\\
		&=\frac{(xx_1x_2)^{-n}(x_1x_2)^{-1}}{n^q}\left(\ti_{p_1}(x_1)-t_n(p_1;x_1)\right)\left(\ti_{p_2}(x_2)-t_n(p_2;x_2)\right)\quad (n\geq 1),\\
		&\Res\left(G_{p_1,p_2,q}(x,x_1,x_2;s),-n-1/2\right)\\
		&=\frac1{(p-1)!} \lim_{s\rightarrow -n-1/2} \frac{d^{p-1}}{ds^{p-1}}\left((s+n+1/2)^pG_{p_1,p_2,q}(x,x_1,x_2;s)\right)\quad (n\geq 0)\\
		&=(-1)^q\sum_{k=0}^{p_1+p_2-1} \binom{p_1+p_2+q-k-2}{q-1} \left((-1)^k \ti_{k+1}(x)-x \ti_{k+1}\Big(x^{-1}\Big) \right) \frac{(xx_1x_2)^n}{(n+1/2)^{p_1+p_2+q-k-1}}\\
		&\quad+(-1)^q\sum_{\sigma\in\left\{(12),(21)\right\}}\sum_{0\le k_1+k_2\le p_{\sigma(1)}-1}\binom{p_{\sigma(1)}+q-k_1-k_2-2}{q-1}\binom{k_2+p_{\sigma(2)}-1}{p_{\sigma(2)}-1}\\
		&\quad\quad\times\left((-1)^{k_1}\ti_{k_1+1}(x)-x\ti_{k_1+1}\Big(x^{-1}\Big)\right)\\
		&\quad\quad\times\left((-1)^{k_2}\Li_{k_2+p_{\sigma(2)}}\big(x_{\sigma(2)}\big)+(-1)^{p_{\sigma(2)}}\zeta_n\Big(k_2+p_{\sigma(2)};x_{\sigma(2)}^{-1}\Big)\right)\frac{(xx_1x_2)^n}{(n+1/2)^{q+p_{\sigma(1)}-k_1-k_2-1}}
	\end{align*}
	and
	\begin{align*}
		&\Res\left(G_{p_1,p_2,q}(x,x_1,x_2;s),0\right)\\
		&=\frac1{q!} \lim_{s\rightarrow 0} \frac{d^{q}}{ds^{q}}\left(s^{q+1}G_{p_1,p_2,q}(x,x_1,x_2;s)\right)\\
		&=\sum_{k_1+k_2=q,\atop k_1,k_2\geq 0}(-1)^q\binom{p_1+k_1-1}{p_1-1}\binom{p_2+k_2-1}{p_2-1}(x_1x_2)^{-1}\ti_{p_1+k_1}(x_1)\ti_{p_2+k_2}(x_2)\\
		&\quad+\sum_{k_1+k_2+k_3=q-1,\atop k_1,k_2,k_3\geq 0}\left((-1)^{k_1}\Li_{k_1+1}(x)-\Li_{k_1+1}\Big(x^{-1}\Big)\right)\\
		&\quad\quad\times(-1)^{k_2+k_3}\binom{k_2+p_1-1}{p_1-1}\binom{k_3+p_2-1}{p_2-1} \ti_{k_2+p_1}(x_1)\ti_{k_3+p_2}(x_2)(x_1x_2)^{-1}.
	\end{align*}
	Applying Lemma \ref{lem-redisue-thm} in conjunction with the aforementioned four residue values suffices to prove the theorem.
\end{proof}
\begin{exa}Setting $(p_1,p_2,q)=(1,1,2)$ in Theorem \ref{thm-unlinearCES-2}, we have
\begin{align*}
	&(x_1x_2)^{-1}\tilde{S}_{1,1;2}\Big(x_1,x_2;(xx_1x_2)^{-1}\Big)+ \tilde{S}_{1,1;2}\Big(x_1^{-1},x_2^{-1};xx_1x_2\Big)\nonumber\\
	&=-(x_1x_2)^{-1}\ti_{1}(x_1)\ti_{1}(x_2)\Li_{2}\Big((xx_1x_2)^{-1}\Big)-(x_1x_2)^{-1}\ti_{1}(x_1)\ti_{1}(x_2)\Li_{2}(xx_1x_2)\nonumber\\
	&\quad+(x_1x_2)^{-1}\ti_{1}\big(x_{1}\big)\tilde{S}_{1;2}\Big(x_{2};(xx_1x_2)^{-1}\Big)+(x_1x_2)^{-1}\ti_{1}\big(x_{2}\big)\tilde{S}_{1;2}\Big(x_{1};(xx_1x_2)^{-1}\Big)\\
	&\quad+x_{1}^{-1}\ti_{1}\big(x_{1}\big)\tilde{S}_{1;2}\Big(x_{2}^{-1};xx_1x_2\Big)+x_{2}^{-1}\ti_{1}\big(x_{2}\big)\tilde{S}_{1;2}\Big(x_{1}^{-1};xx_1x_2\Big)\\
	&\quad-2(xx_1x_2)^{-1}\left(\ti_{1}(x)-x \ti_{1}\Big(x^{-1}\Big) \right) \ti_{3}(xx_1x_2)+(xx_1x_2)^{-1}\left(\ti_{2}(x)+x \ti_{2}\Big(x^{-1}\Big)\right)\ti_{2}(xx_1x_2)\\
	&\quad-\left(\ti_{1}(x)-x\ti_{1}\Big(x^{-1}\Big)\right)
	\left((xx_1x_2)^{-1}\Li_{1}(x_{2})\ti_{2}(xx_1x_2)-R_{1;2}\Big(x_{2}^{-1};(xx_1x_2)\Big)\right)\\
	&\quad-\left(\ti_{1}(x)-x\ti_{1}\Big(x^{-1}\Big)\right)\left(
		(xx_1x_2)^{-1}\Li_{1}(x_{1})\ti_{2}(xx_1x_2)-R_{1;2}\Big(x_{1}^{-1};(xx_1x_2)\Big)\right)\\
	&\quad-(x_1x_2)^{-1}\ti_{3}(x_1)\ti_{1}(x_2)-(x_1x_2)^{-1}\ti_{1}(x_1)\ti_{3}(x_2)-(x_1x_2)^{-1}\ti_{2}(x_1)\ti_{2}(x_2)\nonumber\\
	&\quad+\left(\Li_{2}(x)+\Li_{2}\Big(x^{-1}\Big)\right) \ti_{1}(x_1)\ti_{1}(x_2)(x_1x_2)^{-1}+\left(\Li_{1}(x)-\Li_{1}\Big(x^{-1}\Big)\right) \ti_{1}(x_1)\ti_{2}(x_2)(x_1x_2)^{-1}\\
	&\quad+\left(\Li_{1}(x)-\Li_{1}\Big(x^{-1}\Big)\right) \ti_{2}(x_1)\ti_{1}(x_2)(x_1x_2)^{-1}.
\end{align*}
Setting $(p_1,p_2,q)=(1,2,2)$ in Theorem \ref{thm-unlinearCES-2}, we have
	\begin{align*}
	&(x_1x_2)^{-1}\tilde{S}_{1,2;2}\Big(x_1,x_2;(xx_1x_2)^{-1}\Big)- \tilde{S}_{1,2;2}\Big(x_1^{-1},x_2^{-1};xx_1x_2\Big)\nonumber\\
	&=-(x_1x_2)^{-1}\ti_{1}(x_1)\ti_{2}(x_2)\Li_{2}\Big((xx_1x_2)^{-1}\Big)-(x_1x_2)^{-1}\ti_{1}(x_1)\ti_{2}(x_2)\Li_{2}(xx_1x_2)\nonumber\\
	&\quad+(x_1x_2)^{-1}\ti_{1}\big(x_{1}\big)\tilde{S}_{2;2}\Big(x_{2};(xx_1x_2)^{-1}\Big)+(x_1x_2)^{-1}\ti_{2}\big(x_{2}\big)\tilde{S}_{1;2}\Big(x_{1};(xx_1x_2)^{-1}\Big)\\
	&\quad-x_{1}^{-1}\ti_{1}\big(x_{1}\big)\tilde{S}_{2;2}\Big(x_{2}^{-1};xx_1x_2\Big)+x_{2}^{-1}\ti_{2}\big(x_{2}\big)\tilde{S}_{1;2}\Big(x_{1}^{-1};xx_1x_2\Big)\\
	&\quad-3(xx_1x_2)^{-1}\left(\ti_{1}(x)-x \ti_{1}\Big(x^{-1}\Big) \right) \ti_{4}(xx_1x_2)+2(xx_1x_2)^{-1}\left(\ti_{2}(x)+x \ti_{2}\Big(x^{-1}\Big) \right) \ti_{3}(xx_1x_2)\\
	&\quad-(xx_1x_2)^{-1}\left(\ti_{3}(x)-x \ti_{3}\Big(x^{-1}\Big) \right) \ti_{2}(xx_1x_2)\nonumber\\
	&\quad-\left(\ti_{1}(x)-x\ti_{1}\Big(x^{-1}\Big)\right)\left((xx_1x_2)^{-1}\Li_{2}\big(x_{2}\big)\ti_{2}(xx_1x_2)+R_{2;2}\Big(x_{2}^{-1};(xx_1x_2)\Big)\right)\\
	&\quad-2\left(\ti_{1}(x)-x\ti_{1}\Big(x^{-1}\Big)\right)\left((xx_1x_2)^{-1}\Li_{1}\big(x_{1}\big)\ti_{3}(xx_1x_2)-R_{1;3}\Big(x_{1}^{-1};(xx_1x_2)\Big)\right)\\
	&\quad+\left(\ti_{2}(x)+x\ti_{2}\Big(x^{-1}\Big)\right)\left((xx_1x_2)^{-1}\Li_{1}\big(x_{1}\big)\ti_{2}(xx_1x_2)-R_{1;2}\Big(x_{1}^{-1};(xx_1x_2)\Big)\right)\\
	&\quad+\left(\ti_{1}(x)-x\ti_{1}\Big(x^{-1}\Big)\right)\left((xx_1x_2)^{-1}\Li_{2}\big(x_{1}\big)\ti_{2}(xx_1x_2)+R_{2;2}\Big(x_{1}^{-1};(xx_1x_2)\Big)\right)\\
	&\quad-3(x_1x_2)^{-1}\ti_{1}(x_1)\ti_{4}(x_2)-(x_1x_2)^{-1}\ti_{3}(x_1)\ti_{2}(x_2)-2(x_1x_2)^{-1}\ti_{2}(x_1)\ti_{3}(x_2)\\	
	&\quad+\left(\Li_{2}(x)+\Li_{2}\Big(x^{-1}\Big)\right) \ti_{1}(x_1)\ti_{2}(x_2)(x_1x_2)^{-1}+\left(\Li_{1}(x)-\Li_{1}\Big(x^{-1}\Big)\right) \ti_{2}(x_1)\ti_{2}(x_2)(x_1x_2)^{-1}\\
	&\quad+2\left(\Li_{1}(x)-\Li_{1}\Big(x^{-1}\Big)\right) \ti_{1}(x_1)\ti_{3}(x_2)(x_1x_2)^{-1}.
\end{align*}
\end{exa}
In \cite{Wang-Xu2025}, the author of this paper and Wang presented the following relationships between cyclotomic linear and quadratic Euler $\tilde{S}$-sums and cyclotomic multiple $T$-values of depth $\leq 3$:
\begin{align*}
&\Ti_{p,q}(x,y)=\frac{y}{2^{p+q-2}}\tilde{S}_{p;q}(x;y),\\
&\Ti_{p,q,r}(x,y,z)=\frac{yz}{2^{p+q+r-3}}\ti_r(z)\tilde{S}_{p;q}(x;y)-\frac{yz}{2^{p+q+r-3}}\tilde{S}_{p,r;q}(x,z;y).
\end{align*}
They also provided declarative results on the parity of cyclotomic multiple $T$-values of depth $\leq 3$ but did not give explicit formulas. However, Theorems \ref{thm-linearCES-2} and \ref{thm-unlinearCES-2} in this paper can supply specific formulas for the parity of cyclotomic multiple $T$-values of depth $\leq 3$, which will not be elaborated further here.

\section{Parity of Arbitrary-Order Cyclotomic Euler $R/\tilde{S}$-Sums}
In this section, we will employ the method of contour integrals to prove Theorems \ref{thm-parityc-CERSmain} and \ref{thm-parityc-CE-SS}.
\begin{thm}\label{thm-parityc-CERS} Let $x,x_1,\ldots,x_r$ be roots of unity, and $p_1,\ldots,p_r,q\geq 1$ with $(p_j,x_j)\neq (1,1)\ (j=1,2,\ldots,r)$ and $ (q,xx_1\cdots x_r)\neq (1,1)$. The
\begin{align*}
&(xx_1\cdots x_r)^{-1}R_{p_1,\ldots,p_r;q}\Big(x_1,\ldots,x_r;(xx_1\cdots x_r)^{-1}\Big)\\
&\quad+(-1)^{p_1+\cdots+p_r+q+r}R_{p_1,\ldots,p_r;q}\Big(x_1^{-1},\ldots,x_r^{-1};xx_1\cdots x_r\Big)
\end{align*}
reduces to a combination of sums of lower orders.
\end{thm}
\begin{proof}
To prove this theorem, we only need to consider contour integrals of the following form:
\begin{align*}
&\oint\limits_{\left( \infty  \right)}H_{p_1,p_2,\cdots,p_r,q}(x,x_1,x_2,\ldots,x_r;s)ds\\
&:= \oint\limits_{\left( \infty  \right)}\frac{\Phi(s;x)\phi^{(p_1-1)}(s;x_1)\phi^{(p_2-1)}(s;x_2)\cdots\phi^{(p_r-1)}(s;x_r)}{(p_1-1)!(p_2-1)!\cdots (p_r-1)!(s-1/2)^q} (-1)^{p_1+p_2+\cdots+p_r-r}ds=0.
\end{align*}
Obviously, all positive integers are simple poles, all non-positive integers are poles of order $p_1+p_2+\cdots+p_r+1$, and $s=1/2$ is a pole of order $q$. Applying Lemma \ref{lem-redisue-thm}, we have
\begin{align*}
&\sum_{n=1}^\infty \Res\left(H_{p_1,p_2,\ldots,p_r,q}(\cdot;s),n\right)+\sum_{n=0}^\infty \Res\left(H_{p_1,p_2,\ldots,p_r,q}(\cdot;s),-n\right)+\Res\left(H_{p_1,p_2,\ldots,p_r,q}(\cdot;s),1/2\right)=0.
\end{align*}
By direct residue calculations, we find that
\begin{align*}
&\sum_{n=1}^\infty \Res\left(H_{p_1,p_2,\ldots,p_r,q}(\cdot;s),n\right)+\sum_{n=0}^\infty \Res\left(H_{p_1,p_2,\ldots,p_r,q}(\cdot;s),-n\right)\\
&=(xx_1\cdots x_r)^{-1}(-1)^rR_{p_1,\ldots,p_r;q}\Big(x_1,\ldots,x_r;(xx_1\cdots x_r)^{-1}\Big)\\
&\quad+(-1)^{p_1+\cdots+p_r+q}R_{p_1,\ldots,p_r;q}\Big(x_1^{-1},\ldots,x_r^{-1};xx_1\cdots x_r\Big)\\
&\quad+\{\text{combinations of lower-order cyclotomic Euler $R$-sums}\}
\end{align*}
and
\begin{align*}
\Res\left(H_{p_1,p_2,\ldots,p_r,q}(\cdot;s),1/2\right)\in \{\text{combinations of lower-order cyclotomic Euler $R$-sums}\}.
\end{align*}
Therefore, combining the two residue results above completes the proof of the theorem.
\end{proof}

\begin{thm}\label{thm-parityc-CESS} Let $x,x_1,\ldots,x_r$ be roots of unity, and $p_1,\ldots,p_r,q\geq 1$ with $(p_j,x_j)\neq (1,1)\ (j=1,2,\ldots,r)$ and $ (q,xx_1\cdots x_r)\neq (1,1)$. The
\begin{align*}
&(x_1\cdots x_r)^{-1}\tilde{S}_{p_1,p_2,\ldots,p_r;q}\Big(x_1,x_2,\ldots,x_r;(xx_1\cdots x_r)^{-1}\Big)\\&\quad+(-1)^{p_1+p_2+\cdots+p_r+q+r}\tilde{S}_{p_1,p_2,\ldots,p_r;q}\Big(x_1^{-1},x_2^{-1},\ldots,x_r^{-1};xx_1\cdots x_r\Big)
\end{align*}
reduces to a combination of sums of lower orders (It should be noted that the lower-order sum also contains Euler $R$-sums).
\end{thm}
\begin{proof}
To prove the theorem, it suffices to consider the residue computations of the following contour integral:
\begin{align*}
&(-1)^{p_1+p_2+\cdots+p_r-r}\oint\limits_{\left( \infty  \right)}G_{p_1p_2\cdots p_r,q}(x,x_1,x_2,\ldots,x_r;s)ds\\
&:= \oint\limits_{\left( \infty  \right)}\frac{\Phi(s;x)\phi^{(p_1-1)}(s+1/2;x_1)\phi^{(p_2-1)}(s+1/2;x_2)\cdots\phi^{(p_r-1)}(s+1/2;x_r)}{(p_1-1)!(p_2-1)!\cdots (p_r-1)!s^q} ds=0.
\end{align*}
The specific procedure is similar to the proof of Theorem \ref{thm-parityc-CERS} and is therefore omitted here.
\end{proof}

\begin{re}
Theorems \ref{thm-parityc-CERSmain} and \ref{thm-parityc-CE-SS} are obtained by replacing $x$ with $(xx_1\cdots x_r)^{-1}$ in Theorems \ref{thm-parityc-CERS} and \ref{thm-parityc-CESS}, respectively.
\end{re}

\section{Further Remark}
The results in the preceding sections were obtained by considering contour integrals involving functions composed of rational functions of the form $r(s) = 1/s^q$ or $(s - 1/2)^q$ (for these two cases) and kernel functions. In this section, we will examine contour integrals involving functions formed by rational functions $r(s)$ of several general forms and kernel functions. Through residue computations, we will derive several general residue formulas.

Denote by $S_1$, $S_2$ and $S_3$ the set of poles of rational functions $r_1(s)$, $r_2(s)$ and $r_3(s)$.
\begin{thm}Let $x$,$y$ be roots of unity and $p\ge1$. Assuming that $r_1(s)$ has a pole at $0$ of order $q\ge0$, then we have
	\begin{align}
		&-\sum_{\alpha\in S_1\backslash\left\{0\right\}}\Res\left(\frac{\Phi(s;x)\phi^{(p_1-1)}(s;y)}{(p-1)!}(-1)^{p-1}r_1(s),\alpha\right)\nonumber\\
		&=\sum_{n=1}^{\infty}(xy)^{-n}\left(\Li_p(y)-\zeta_{n-1}(p;y)\right)r_1(n)+\sum_{n=1}^{\infty}(xy)^n\frac{r_1^{(p)}(-n)}{p!}\nonumber\\
		&+\sum_{n=1}^{\infty}\sum_{k=0}^{p-1}\left((-1)^k\Li_{k+1}(x)-\Li_{k+1}\Big(x^{-1}\Big)\right)\frac{(xy)^nr_1^{(p-k-1)}(-n)}{(p-k-1)!}\nonumber\\
		&+\sum_{n=1}^{\infty}\left((-1)^k\Li_{p}(y)+(-1)^p\zeta_{n}\Big(p;y^{-1}\Big)\right)(xy)^nr_1(-n)+\frac{R_1^{(p+q)}(0)}{(p+q)!}\nonumber\\
		&+\sum_{k=0}^{q}\binom{k+p-1}{p-1}(-1)^k\Li_{k+p}(y)\frac{R_1^{(q-k)}(0)}{(q-k)!}\nonumber\\
		&+\sum_{k=0}^{p+q-1}\left((-1)^k\Li_{k+1}(x)-\Li_{k+1}\Big(x^{-1}\Big)\right)\frac{R_1^{(p+q-k-1)}(0)}{(p+q-k-1)!}\nonumber\\
		&+\sum_{0\le k_1+k_2\le q-1}\left((-1)^{k_1}\Li_{k_1+1}(x)-\Li_{k_1+1}\Big(x^{-1}\Big)\right)\nonumber\\
		&\quad\times\binom{k_2+p-1}{p-1}(-1)^{k_2}\Li_{k_2+p}(y)\frac{R_1^{(q-k_1-k_2-1)}(0)}{(q-k_1-k_2-1)!},
	\end{align}
	where $R_1(s)=s^{q}r_1(s)$.
\end{thm}
\begin{proof}First, considering the residue calculations of the following contour integral:
	\begin{align*}
		\oint\limits_{\left( \infty  \right)} h_{p,q}(x,y;s)ds:= \oint\limits_{\left( \infty  \right)} \frac{\Phi(s;x)\phi^{(p-1)}(s;y)}{(p-1)!}r_1(s) (-1)^{p-1}ds=0.
	\end{align*}
The integrand $h_{p,q}(x,y;s)$ has the following poles throughout the complex plane: $n$ (for postive
integer $n$, simple poles), $-n$ (for nonnegative
integer $n$, poles of order $p+1$) and $0$ (pole of order $p+q+1$). By direct calculations, we deduce the
following residues:
\begin{align*}
	&\Res\left(h_{p,q}(x,y;s),n\right)=(xy)^{-n}\left(\Li_p(y)-\zeta_{n-1}(p;y)\right)r_1(n)\quad(n\ge1),\\
	&\Res\left(h_{p,q}(x,y;s),-n\right)=(xy)^n\frac{r_1^{(p)}(-n)}{p!}+\sum_{k=0}^{p-1}\left((-1)^k\Li_{k+1}(x)-\Li_{k+1}\Big(x^{-1}\Big)\right)\frac{(xy)^nr_1^{(p-k-1)}(-n)}{(p-k-1)!}\\
	&\qquad\qquad\qquad\qquad\qquad+\left((-1)^k\Li_{p}(y)+(-1)^p\zeta_{n}\Big(p;y^{-1}\Big)\right)(xy)^nr_1(-n)\quad(n\ge1),\\
	&\Res\left(h_{p,q}(x,y;s),0\right)=\frac{R_1^{(p+q)}(0)}{(p+q)!}+\sum_{k=0}^{q}\binom{k+p-1}{p-1}(-1)^k\Li_{k+p}(y)\frac{R_1^{(q-k)}(0)}{(q-k)!}\\
	&\quad\qquad\qquad\qquad\qquad+\sum_{k=0}^{p+q-1}\left((-1)^k\Li_{k+1}(x)-\Li_{k+1}\Big(x^{-1}\Big)\right)\frac{R_1^{(p+q-k-1)}(0)}{(p+q-k-1)!}\\
	&\quad\qquad\qquad\qquad\qquad+\sum_{0\le k_1+k_2\le q-1}\left((-1)^{k_1}\Li_{k_1+1}(x)-\Li_{k_1+1}\Big(x^{-1}\Big)\right)\\
	&\qquad\qquad\qquad\qquad\qquad\times\binom{k_2+p-1}{p-1}(-1)^{k_2}\Li_{k_2+p}(y)\frac{R_1^{(q-k_1-k_2-1)}(0)}{(q-k_1-k_2-1)!}.
\end{align*}
From Lemma \ref{lem-redisue-thm}, we know that
\begin{align*}
	\Res\left(h_{p,q}(x,y;s),0\right)&+\sum_{n=1}^{\infty}\Res\left(h_{p,q}(x,y;s),n\right)+\sum_{n=1}^{\infty}\Res\left(h_{p,q}(x,y;s),-n\right)\\&+\sum_{\alpha\in S_1\backslash\left\{0\right\}}\Res\left(h_{p,q}(x,y;s),\alpha\right)=0.
\end{align*}
Finally, combining these contributions yields the result.
\end{proof}
\begin{thm}Let $x$,$x_1$,$x_2$ be roots of unity and $p_1,p_2\ge1$. Assuming that $r_1(s)$ has a pole at $0$ of order $q\ge0$, then we have
	\begin{align}
		&-\sum_{\alpha\in S_1\backslash\left\{0\right\}}\Res\left(\frac{\Phi(s;x)\phi^{(p_1-1)}(s;x_1)\phi^{(p_2-1)}(s;x_2)}{(p_1-1)!(p_2-1)!}(-1)^{p_1+p_2}r_1(s),\alpha\right)\nonumber\\
		&=\sum_{n=1}^{\infty}(xx_1x_2)^{-n}\left(\Li_{p_1}(x_1)-\zeta_{n-1}(p_1;x_1)\right)\left(\Li_{p_2}(x_2)-\zeta_{n-1}(p_2;x_2)\right)r_1(n)+\sum_{n=1}^{\infty}(xx_1x_2)^n\frac{r_1^{(p_1+p_2)}(-n)}{(p_1+p_2)!}\nonumber\\
	    &\quad+\sum_{\sigma\in\left\{(12),(21)\right\}}\sum_{n=1}^{\infty}\sum_{k=0}^{p_{\sigma(1)}}\binom{k+p_{\sigma(2)}-1}{p_{\sigma(2)}-1}(xx_1x_2)^n\nonumber\\
	    &\quad\quad\times\left((-1)^k\Li_{k+p_{\sigma(2)}}\big(x_{\sigma(2)}\big)+(-1)^{p_{\sigma(2)}}\zeta_{n}\Big(k+p_{\sigma(2)};x_{\sigma(2)}^{-1}\Big)\right)\frac{r_1^{(p_{\sigma(1)}-k)}(-n)}{\big(p_{\sigma(1)}-k\big)!}\nonumber\\
		&\quad+\sum_{n=1}^{\infty}(xx_1x_2)^n\left(\Li_{p_1}(x_1)+(-1)^{p_1}\zeta_n\Big(p_1;x_1^{-1}\Big)\right)\left(\Li_{p_2}(x_2)+(-1)^{p_2}\zeta_n\Big(p_2;x_2^{-1}\Big)\right)r_1(-n)\nonumber\\
		&\quad+\sum_{\sigma\in\left\{(12),(21)\right\}}\sum_{n=1}^{\infty}\sum_{0\le k_1+k_2\le p_{\sigma(1)}-1,\atop k_1,k_2\ge0}\left((-1)^{k_1}\Li_{k_1+1}(x)-\Li_{k_1+1}\Big(x^{-1}\Big)\right)\binom{k_2+p_{\sigma(2)}-1}{p_{\sigma(2)}-1}\nonumber\\
		&\quad\quad\times\left((-1)^{k_2}\Li_{k_2+p_{\sigma(2)}}\big(x_{\sigma(2)}\big)+(-1)^{p_{\sigma(2)}}\zeta_n\Big(k_2+p_{\sigma(2)};x_{\sigma(2)}^{-1}\Big)\right)\frac{(xx_1x_2)^nr_1^{(p_{\sigma(1)}-k_1-k_2-1)}(-n)}{\big(p_{\sigma(1)}-k_1-k_2-1\big)!}\nonumber\\
		&\quad+\sum_{n=1}^{\infty}\sum_{k=0}^{p_1+p_2-1}\left((-1)^k\Li_{k+1}(x)-\Li_{k+1}\Big(x^{-1}\Big)\right)\frac{(xx_1x_2)^nr_1^{(p_1+p_2-k-1)}(-n)}{(p_1+p_2-k-1)!}\nonumber\\
		&\quad+\frac{R_1^{(p_1+p_2+q)}(0)}{(p_1+p_2+q)!}+\sum_{\sigma\in\left\{(12),(21)\right\}}\sum_{k=0}^{p_{\sigma(1)}+q}\binom{k+p_{\sigma(2)}-1}{p_{\sigma(2)}-1}(-1)^k\Li_{k+p_{\sigma(2)}}\big(x_{\sigma(2)}\big)\frac{R_1^{(p_{\sigma(1)}+q-k)}(0)}{\big(p_{\sigma(1)}+q-k\big)!}\nonumber\\
		&\quad+\sum_{0\le k_1+k_2\le q,\atop k_1,k_2\ge0}\binom{k_1+p_1-1}{p_1-1}\binom{k_2+p_2-1}{p_2-1}(-1)^{k_1+k_2}\Li_{k_1+p_1}(x_1)\Li_{k_2+p_2}(x_2)\frac{R_1^{(q-k_1-k_2)}(0)}{(q-k_1-k_2)!}\nonumber\\
		&\quad+\sum_{k=0}^{p_1+p_2+q-1}\left((-1)^k\Li_{k+1}(x)-\Li_{k+1}\Big(x^{-1}\Big)\right)\frac{R_1^{(p_1+p_2+q-k-1)}(0)}{(p_1+p_2+q-k-1)!}\nonumber\\
		&\quad+\sum_{\sigma\in\left\{(12),(21)\right\}}\sum_{0\le k_1+k_2\le p_{\sigma(1)}+q-1,\atop k_1,k_2\ge0}\left((-1)^{k_1}\Li_{k_1+1}(x)-\Li_{k_1+1}\Big(x^{-1}\Big)\right)\nonumber\\
		&\quad\quad\times\binom{k_2+p_{\sigma(2)}-1}{p_{\sigma(2)}-1}(-1)^{k_2}\Li_{k_2+p_{\sigma(2)}}\big(x_{\sigma(2)}\big)\frac{R_1^{(p_{\sigma(1)}+q-k_1-k_2-1)}(0)}{\big(p_{\sigma(1)}+q-k_1-k_2-1\big)!}\nonumber\\
		&\quad+\sum_{0\le k_1+k_2+k_3\le q-1,\atop k_1,k_2,k_3\ge0}\left((-1)^{k_1}\Li_{k_1+1}(x)-\Li_{k_1+1}\Big(x^{-1}\Big)\right)\nonumber\\
		&\quad\quad\times\binom{k_2+p_1-1}{p_1-1}\binom{k_3+p_2-1}{p_2-1}(-1)^{k_2+k_3}\Li_{k_2+p_1}(x_1)\Li_{k_3+p_2}(x_2)\frac{R_1^{(q-k_1-k_2-k_3-1)}(0)}{(q-k_1-k_2-k_3-1)!},
	\end{align}
	where $R_1(s)=s^{q}r_1(s)$.
\end{thm}
\begin{proof}First, considering the residue calculations of the following contour integral:
	\begin{align*}
		\oint\limits_{\left( \infty  \right)} h_{p_1,p_2,q}(x,x_1,x_2;s)ds:= \oint\limits_{\left( \infty  \right)} \frac{\Phi(s;x)\phi^{(p_1-1)}(s;x_1)\phi^{(p_2-1)}(s;x_2)}{(p_1-1)!(p_2-1)!}r_1(s) (-1)^{p_1+p_2}ds=0.
	\end{align*}
	The integrand $h_{p_1,p_2,q}(x,x_1,x_2;s)$ has the following poles throughout the complex plane: $n$ (for postive
	integer $n$, simple poles), $-n$ (for nonnegative
	integer $n$, poles of order $p_1+p_2+1$) and $0$ (pole of order $p_1+p_2+q+1$). By direct calculations, we deduce the following residues:
	\begin{align*}
		&\Res\left(h_{p_1,p_2,q}(x,x_1,x_2;s),n\right)\\
		&=(xx_1x_2)^{-n}\left(\Li_{p_1}(x_1)-\zeta_{n-1}(p_1;x_1)\right)\left(\Li_{p_2}(x_2)-\zeta_{n-1}(p_2;x_2)\right)r_1(n)\quad(n\ge1),\\
		&\Res\left(h_{p_1,p_2,q}(x,x_1,x_2;s),-n\right)\\
		&=(xx_1x_2)^n\frac{r_1^{(p_1+p_2)}(-n)}{(p_1+p_2)!}\\
		&\quad+\sum_{k=0}^{p_1}\binom{k+p_2-1}{p_2-1}\left((-1)^k\Li_{k+p_2}(x_2)+(-1)^{p_2}\zeta_{n}\Big(k+p_2;x_2^{-1}\Big)\right)(xx_1x_2)^n\frac{r_1^{(p_1-k)}(-n)}{(p_1-k)!}\\
		&\quad+\sum_{k=0}^{p_2}\binom{k+p_1-1}{p_1-1}\left((-1)^k\Li_{k+p_1}(x_1)+(-1)^{p_1}\zeta_{n}\Big(k+p_1;x_1^{-1}\Big)\right)(xx_1x_2)^n\frac{r_1^{(p_2-k)}(-n)}{(p_2-k)!}\\
		&\quad+(xx_1x_2)^n\left(\Li_{p_1}(x_1)+(-1)^{p_1}\zeta_n\Big(p_1;x_1^{-1}\Big)\right)\left(\Li_{p_2}(x_2)+(-1)^{p_2}\zeta_n\Big(p_2;x_2^{-1}\Big)\right)r_1(-n)\\
		&\quad+\sum_{0\le k_1+k_2\le p_2-1,\atop k_1,k_2\ge0}\left((-1)^{k_1}\Li_{k_1+1}(x)-\Li_{k_1+1}\Big(x^{-1}\Big)\right)\\
		&\quad\quad\times\binom{k_2+p_1-1}{p_1-1}\left((-1)^{k_2}\Li_{k_2+p_1}(x_1)+(-1)^{p_1}\zeta_n\Big(k_2+p_1;x_1^{-1}\Big)\right)\frac{(xx_1x_2)^nr_1^{(p_2-k_1-k_2-1)}(-n)}{(p_2-k_1-k_2-1)!}\\
		&\quad+\sum_{0\le k_1+k_2\le p_1-1,\atop k_1,k_2\ge0}\left((-1)^{k_1}\Li_{k_1+1}(x)-\Li_{k_1+1}\Big(x^{-1}\Big)\right)\\
		&\quad\quad\times\binom{k_2+p_2-1}{p_2-1}\left((-1)^{k_2}\Li_{k_2+p_2}(x_2)+(-1)^{p_2}\zeta_n\Big(k_2+p_2;x_2^{-1}\Big)\right)\frac{(xx_1x_2)^nr_1^{(p_1-k_1-k_2-1)}(-n)}{(p_1-k_1-k_2-1)!}\\
		&\quad+\sum_{k=0}^{p_1+p_2-1}\left((-1)^k\Li_{k+1}(x)-\Li_{k+1}\Big(x^{-1}\Big)\right)\frac{(xx_1x_2)^nr_1^{(p_1+p_2-k-1)}(-n)}{(p_1+p_2-k-1)!}\quad(n\ge1),\\
		&\Res\left(h_{p_1,p_2,q}(x,x_1,x_2;s),0\right)\\
		&=\frac{R_1^{(p_1+p_2+q)}(0)}{(p_1+p_2+q)!}+\sum_{k=0}^{p_2+q}\binom{k+p_1-1}{p_1-1}(-1)^k\Li_{k+p_1}(x_1)\frac{R_1^{(p_2+q-k)}(0)}{(p_2+q-k)!}\\
		&\quad+\sum_{k=0}^{p_1+q}\binom{k+p_2-1}{p_2-1}(-1)^k\Li_{k+p_2}(x_2)\frac{R_1^{(p_1+q-k)}(0)}{(p_1+q-k)!}\\
		&\quad+\sum_{0\le k_1+k_2\le q,\atop k_1,k_2\ge0}\binom{k_1+p_1-1}{p_1-1}\binom{k_2+p_2-1}{p_2-1}(-1)^{k_1+k_2}\Li_{k_1+p_1}(x_1)\Li_{k_2+p_2}(x_2)\frac{R_1^{(q-k_1-k_2)}(0)}{(q-k_1-k_2)!}\\
		&\quad+\sum_{k=0}^{p_1+p_2+q-1}\left((-1)^k\Li_{k+1}(x)-\Li_{k+1}\Big(x^{-1}\Big)\right)\frac{R_1^{(p_1+p_2+q-k-1)}(0)}{(p_1+p_2+q-k-1)!}\\
		&\quad+\sum_{0\le k_1+k_2\le p_2+q-1,\atop k_1,k_2\ge0}\left((-1)^{k_1}\Li_{k_1+1}(x)-\Li_{k_1+1}\Big(x^{-1}\Big)\right)\\
		&\quad\quad\times\binom{k_2+p_1-1}{p_1-1}(-1)^{k_2}\Li_{k_2+p_1}(x_1)\frac{R_1^{(p_2+q-k_1-k_2-1)}(0)}{(p_2+q-k_1-k_2-1)!}\\
		&\quad+\sum_{0\le k_1+k_2\le p_1+q-1,\atop k_1,k_2\ge0}\left((-1)^{k_1}\Li_{k_1+1}(x)-\Li_{k_1+1}\Big(x^{-1}\Big)\right)\\
		&\quad\quad\times\binom{k_2+p_2-1}{p_2-1}(-1)^{k_2}\Li_{k_2+p_2}(x_2)\frac{R_1^{(p_1+q-k_1-k_2-1)}(0)}{(p_1+q-k_1-k_2-1)!}\\
		&\quad+\sum_{0\le k_1+k_2+k_3\le q-1,\atop k_1,k_2,k_3\ge0}\left((-1)^{k_1}\Li_{k_1+1}(x)-\Li_{k_1+1}\Big(x^{-1}\Big)\right)\\
		&\quad\quad\times\binom{k_2+p_1-1}{p_1-1}\binom{k_3+p_2-1}{p_2-1}(-1)^{k_2+k_3}\Li_{k_2+p_1}(x_1)\Li_{k_3+p_2}(x_2)\frac{R_1^{(q-k_1-k_2-k_3-1)}(0)}{(q-k_1-k_2-k_3-1)!}.
	\end{align*}
	From Lemma \ref{lem-redisue-thm}, we know that
	\begin{align*}
		&\Res\left(h_{p_1,p_2,q}(x,x_1,x_2;s),0\right)+\sum_{n=1}^{\infty}\Res\left(h_{p_1,p_2,q}(x,x_1,x_2;s),n\right)\\&+\sum_{n=1}^{\infty}\Res\left(h_{p_1,p_2,q}(x,x_1,x_2;s),-n\right)+\sum_{\alpha\in S_1\backslash\left\{0\right\}}\Res\left(h_{p_1,p_2,q}(x,x_1,x_2;s),\alpha\right)=0.
	\end{align*}
	Finally, combining these contributions yields the result.
\end{proof}
\begin{thm}Let $x$,$x_1$,$x_2$ be roots of unity and $p_1,p_2\ge1$. Assuming that $r_2(s)$ has a pole at $-1/2$ of order $q_1\ge0$ and $r_3(s)$ has a pole at 0 of order $q_2\ge0$, then we have
	\begin{align}
		&-\sum_{\alpha\in S_2\backslash\left\{-1/2\right\}}\Res\left(\frac{\Phi(s;x)\phi^{(p_1-1)}(s+1/2;x_1)\phi^{(p_2-1)}(s+1/2;x_2)}{(p_1-1)!(p_2-1)!}(-1)^{p_1+p_2}r_2(s),\alpha\right)\nonumber\\
		&=\sum_{n=1}^{\infty}(xx_1x_2)^{-n}(x_1x_2)^{-1}\left(\ti_{p_1}(x_1)-t_{n}(p_1;x_1)\right)\left(\ti_{p_2}(x_2)-t_{n}(p_2;x_2)\right)r_2(n)\nonumber\\
		&\quad+\sum_{n=0}^{\infty}(xx_1x_2)^n\left(\ti_{p_1}(x_1)x_1^{-1}+(-1)^{p_1}t_{n}\Big(p_1;x_1^{-1}\Big)\right)\left(\ti_{p_2}(x_2)x_2^{-1}+(-1)^{p_2}t_{n}\Big(p_2;x_2^{-1}\Big)\right)r_2(-n)\nonumber\\
	    &\quad+\sum_{n=1}^{\infty}\sum_{k=0}^{p_1+p_2-1}\left((-1)^k\ti_{k+1}(x)-x\ti_{k+1}\Big(x^{-1}\Big)\right)\frac{(xx_1x_2)^nr_2^{(p_1+p_2-k-1)}(-n-1/2)}{(p_1+p_2-k-1)!}\nonumber\\
		&\quad+\sum_{\sigma\in\left\{(12),(21)\right\}}\sum_{n=1}^{\infty}\sum_{0\le k_1+k_2\le p_{\sigma(1)}-1,\atop k_1,k_2\ge0}\left((-1)^{k_1}\ti_{k_1+1}(x)-\ti_{k_1+1}\Big(x^{-1}\Big)\right)\binom{k_2+p_{\sigma(2)}-1}{p_{\sigma(2)}-1}(xx_1x_2)^n\nonumber\\
		&\quad\quad\times\left((-1)^{k_2}\Li_{k_2+p_{\sigma(2)}}(x_{\sigma(2)})+(-1)^{p_{\sigma(2)}}\zeta_n\Big(k_2+p_{\sigma(2)};x_{\sigma(2)}^{-1}\Big)\right)\frac{r_2^{(p_{\sigma(1)}-k_1-k_2-1)}(-n-1/2)}{(p_{\sigma(1)}-k_1-k_2-1)!}\nonumber\\
		&\quad+\sum_{k=0}^{p_1+p_2+q_1-1}\left((-1)^k\ti_{k+1}(x)-x\ti_{k+1}\Big(x^{-1}\Big)\right)\frac{R_2^{(p_1+p_2+q_1-k-1)}(-1/2)}{(p_1+p_2+q_1-k-1)!}\nonumber\\
		&\quad+\sum_{\sigma\in\left\{(12),(21)\right\}}\sum_{0\le k_1+k_2\le p_{\sigma(1)}+q_1-1,\atop k_1,k_2\ge0}\left((-1)^{k_1}\ti_{k_1+1}(x)-x\ti_{k_1+1}\Big(x^{-1}\Big)\right)\nonumber\\
		&\quad\quad\times\binom{k_2+p_{\sigma(2)}-1}{p_{\sigma(2)}-1}(-1)^{k_2}\Li_{k_2+p_{\sigma(2)}}(x_{\sigma(2)})\frac{R_2^{(p_{\sigma(1)}+q_1-k_1-k_2-1)}(-1/2)}{(p_{\sigma(1)}+q_1-k_1-k_2-1)!}\nonumber\\
		&\quad+\sum_{0\le k_1+k_2+k_3\le q_1-1,\atop k_1,k_2,k_3\ge0}\left((-1)^{k_1}\ti_{k_1+1}(x)-x\ti_{k_1+1}\Big(x^{-1}\Big)\right)\nonumber\\
		&\quad\quad\times\binom{k_2+p_1-1}{p_1-1}\binom{k_3+p_2-1}{p_2-1}(-1)^{k_2+k_3}\Li_{k_2+p_1}(x_1)\Li_{k_3+p_2}(x_2)\frac{R_2^{(q_1-k_1-k_2-k_3-1)}(-1/2)}{(q_1-k_1-k_2-k_3-1)!}
	\end{align}
	and
	\begin{align}
		&-\sum_{\alpha\in S_3\backslash\left\{0\right\}}\Res\left(\frac{\Phi(s;x)\phi^{(p_1-1)}(s+1/2;x_1)\phi^{(p_2-1)}(s+1/2;x_2)}{(p_1-1)!(p_2-1)!}(-1)^{p_1+p_2}r_3(s),\alpha\right)\nonumber\\
		&=\sum_{n=1}^{\infty}(xx_1x_2)^{-n}(x_1x_2)^{-1}\left(\ti_{p_1}(x_1)-t_{n}(p_1;x_1)\right)\left(\ti_{p_2}(x_2)-t_{n}(p_2;x_2)\right)r_3(n)\nonumber\\
		&\quad+\sum_{n=1}^{\infty}(xx_1x_2)^n\left(\ti_{p_1}(x_1)x_1^{-1}+(-1)^{p_1}t_{n}\Big(p_1;x_1^{-1}\Big)\right)\left(\ti_{p_2}(x_2)x_2^{-1}+(-1)^{p_2}t_{n}\Big(p_2;x_2^{-1}\Big)\right)r_3(-n)\nonumber\\
		&\quad+\sum_{n=0}^{\infty}\sum_{k=0}^{p_1+p_2-1}\left((-1)^k\ti_{k+1}(x)-x\ti_{k+1}\Big(x^{-1}\Big)\right)\frac{(xx_1x_2)^nr_3^{(p_1+p_2-k-1)}(-n-1/2)}{(p_1+p_2-k-1)!}\nonumber\\
		&\quad+\sum_{\sigma\in\left\{(12),(21)\right\}}\sum_{n=0}^{\infty}\sum_{0\le k_1+k_2\le p_{\sigma(1)}-1,\atop k_1,k_2\ge0}\left((-1)^{k_1}\ti_{k_1+1}(x)-x\ti_{k_1+1}\Big(x^{-1}\Big)\right)\binom{k_2+p_{\sigma(2)}-1}{p_{\sigma(2)}-1}\nonumber\\
		&\quad\quad\times\left((-1)^{k_2}\Li_{k_2+p_{\sigma(2)}}(x_{\sigma(2)})+(-1)^{p_2}\zeta_n\Big(k_2+p_{\sigma(2)};x_{\sigma(2)}^{-1}\Big)\right)\frac{(xx_1x_2)^nr_3^{(p_{\sigma(1)}-k_1-k_2-1)}(-n-1/2)}{(p_{\sigma(1)}-k_1-k_2-1)!}\nonumber\\
		&\quad+\sum_{0\le k_1+k_2\le q_2,\atop k_1,k_2\ge0}(x_1x_2)^{-1}\binom{k_1+p_1-1}{p_1-1}\binom{k_2+p_2-1}{p_2-1}\ti_{k_1+p_1}(x_1)\ti_{k_2+p_2}(x_2)\frac{(-1)^{k_1+k_2}R_3^{(q_2-k_1-k_2)}(0)}{(q_2-k_1-k_2)!}\nonumber\\
		&\quad+\sum_{0\le k_1+k_2+k_3\le q_2-1,\atop k_1,k_2,k_3\ge0}\left((-1)^{k_1}\Li_{k_1+1}(x)-\Li_{k_1+1}\Big(x^{-1}\Big)\right)\nonumber\\
		&\quad\quad\times\binom{k_2+p_1-1}{p_1-1}\binom{k_3+p_2-1}{p_2-1}(-1)^{k_2+k_3}\ti_{k_2+p_1}(x_1)\ti_{k_3+p_2}(x_2)\frac{(x_1x_2)^{-1}R_3^{(q_2-k_1-k_2-k_3-1)}(0)}{(q_2-k_1-k_2-k_3-1)!}
	\end{align}
	where $R_2(s)=(s+1/2)^{q_1}r_2(s)$ and $R_3(s)=s^{q_2}r_3(s)$.
\end{thm}
\begin{proof}First, considering the residue calculations of the following contour integral:
	\begin{align*}
		\oint\limits_{\left( \infty  \right)} g_{p_1,p_2,q_1}(x,x_1,x_2;s)ds:= \oint\limits_{\left( \infty  \right)} \frac{\Phi(s;x)\phi^{(p_1-1)}(s+1/2;x_1)\phi^{(p_2-1)}(s+1/2;x_2)}{(p_1-1)!(p_2-1)!}r_2(s) (-1)^{p_1+p_2}ds=0.
	\end{align*}
	The integrand $g_{p_1,p_2,q_1}(x,x_1,x_2;s)$ has the following poles throughout the complex plane: All integers (simple poles), $-(n+1/2)$ (for positive integer $n$, poles of order $p_1+p_2$) and $-1/2$ (pole of order $p_1+p_2+q_1$). By direct calculations, we deduce the
	following residues:
	\begin{align*}
		&\Res\left(g_{p_1,p_2,q_1}(x,x_1,x_2;s),n\right)\quad(n\ge1)\\
		&=(xx_1x_2)^{-n}(x_1x_2)^{-1}\left(\ti_{p_1}(x_1)-t_{n}(p_1;x_1)\right)\left(\ti_{p_2}(x_2)-t_{n}(p_2;x_2)\right)r_2(n),\\
		&\Res\left(g_{p_1,p_2,q_1}(x,x_1,x_2;s),-n\right)\quad(n\ge0)\\
		&=(xx_1x_2)^n\left(\ti_{p_1}(x_1)x_1^{-1}+(-1)^{p_1}t_{n}\Big(p_1;x_1^{-1}\Big)\right)\left(\ti_{p_2}(x_2)x_2^{-1}+(-1)^{p_2}t_{n}\Big(p_2;x_2^{-1}\Big)\right)r_2(-n),\\
		&\Res\left(g_{p_1,p_2,q_1}(x,x_1,x_2;s),-n-1/2\right)\quad(n\ge1)\\
		&=\sum_{k=0}^{p_1+p_2-1}\left((-1)^k\ti_{k+1}(x)-x\ti_{k+1}\Big(x^{-1}\Big)\right)\frac{(xx_1x_2)^nr_2^{(p_1+p_2-k-1)}(-n-1/2)}{(p_1+p_2-k-1)!}\nonumber\\
		&\quad+\sum_{0\le k_1+k_2\le p_2-1,\atop k_1,k_2\ge0}\left((-1)^{k_1}\ti_{k_1+1}(x)-x\ti_{k_1+1}\Big(x^{-1}\Big)\right)\binom{k_2+p_1-1}{p_1-1}\nonumber\\
		&\quad\quad\times\left((-1)^{k_2}\Li_{k_2+p_1}(x_1)+(-1)^{p_1}\zeta_n\Big(k_2+p_1;x_1^{-1}\Big)\right)\frac{(xx_1x_2)^nr_2^{(p_2-k_1-k_2-1)}(-n-1/2)}{(p_2-k_1-k_2-1)!}\nonumber\\
		&\quad+\sum_{0\le k_1+k_2\le p_1-1,\atop k_1,k_2\ge0}\left((-1)^{k_1}\ti_{k_1+1}(x)-x\ti_{k_1+1}\Big(x^{-1}\Big)\right)\binom{k_2+p_2-1}{p_2-1}\nonumber\\
		&\quad\quad\times\left((-1)^{k_2}\Li_{k_2+p_2}(x_2)+(-1)^{p_2}\zeta_n\Big(k_2+p_2;x_2^{-1}\Big)\right)\frac{(xx_1x_2)^nr_2^{(p_1-k_1-k_2-1)}(-n-1/2)}{(p_1-k_1-k_2-1)!},\\
		&\Res\left(g_{p_1,p_2,q_1}(x,x_1,x_2;s),-1/2\right)\\
		&=\sum_{k=0}^{p_1+p_2+q_1-1}\left((-1)^k\ti_{k+1}(x)-x\ti_{k+1}\Big(x^{-1}\Big)\right)\frac{R_2^{(p_1+p_2+q_1-k-1)}(-1/2)}{(p_1+p_2+q_1-k-1)!}\nonumber\\
		&\quad+\sum_{0\le k_1+k_2\le p_2+q_1-1,\atop k_1,k_2\ge0}\left((-1)^{k_1}\ti_{k_1+1}(x)-x\ti_{k_1+1}\Big(x^{-1}\Big)\right)\nonumber\\
		&\quad\quad\times\binom{k_2+p_1-1}{p_1-1}(-1)^{k_2}\Li_{k_2+p_1}(x_1)\frac{R_2^{(p_2+q_1-k_1-k_2-1)}(-1/2)}{(p_2+q_1-k_1-k_2-1)!}\nonumber\\
		&\quad+\sum_{0\le k_1+k_2\le p_1+q_1-1,\atop k_1,k_2\ge0}\left((-1)^{k_1}\ti_{k_1+1}(x)-x\ti_{k_1+1}\Big(x^{-1}\Big)\right)\nonumber\\
		&\quad\quad\times\binom{k_2+p_2-1}{p_2-1}(-1)^{k_2}\Li_{k_2+p_2}(x_2)\frac{R_2^{(p_1+q_1-k_1-k_2-1)}(-1/2)}{(p_1+q_1-k_1-k_2-1)!}\nonumber\\
		&\quad+\sum_{0\le k_1+k_2+k_3\le q_1-1,\atop k_1,k_2,k_3\ge0}\left((-1)^{k_1}\ti_{k_1+1}(x)-x\ti_{k_1+1}\Big(x^{-1}\Big)\right)\nonumber\\
		&\quad\quad\times\binom{k_2+p_1-1}{p_1-1}\binom{k_3+p_2-1}{p_2-1}(-1)^{k_2+k_3}\Li_{k_2+p_1}(x_1)\Li_{k_3+p_2}(x_2)\frac{R_2^{(q_1-k_1-k_2-k_3-1)}(-1/2)}{(q_1-k_1-k_2-k_3-1)!}.
	\end{align*}
	From Lemma \ref{lem-redisue-thm}, we know that
	\begin{align*}
		&\Res\left(g_{p_1,p_2,q_1}(x,x_1,x_2;s),-1/2\right)+\sum_{n=1}^{\infty}\Res\left(g_{p_1,p_2,q_1}(x,x_1,x_2;s),n\right)\\&+\sum_{n=1}^{\infty}\Res\left(g_{p_1,p_2,q_1}(x,x_1,x_2;s),-n-1/2\right)+\sum_{n=0}^{\infty}\Res\left(g_{p_1,p_2,q_1}(x,x_1,x_2;s),-n\right)\\&+\sum_{\alpha\in S_2\backslash\left\{-1/2\right\}}\Res\left(g_{p_1,p_2,q_1}(x,x_1,x_2;s),\alpha\right)=0.
	\end{align*}
	Finally, combining these contributions yields the result.\\
	Considering the residue calculations of the following contour integral:
	\begin{align*}
	\oint\limits_{\left( \infty  \right)} g_{p_1,p_2,q_2}(x,x_1,x_2;s)ds:= \oint\limits_{\left( \infty  \right)} \frac{\Phi(s;x)\phi^{(p_1-1)}(s+1/2;x_1)\phi^{(p_2-1)}(s+1/2;x_2)}{(p_1-1)!(p_2-1)!}r_3(s) (-1)^{p_1+p_2}ds=0.
	\end{align*}
		The integrand $g_{p_1,p_2,q_2}(x,x_1,x_2;s)$ has the following poles throughout the complex plane: All integers (simple poles), $-(n+1/2)$ (for positive integer $n$, poles of order $p_1+p_2$) and $0$ (pole of order $q_2+1$). By direct calculations, we deduce the following residues:
	\begin{align*}
		&\Res\left(g_{p_1,p_2,q_2}(x,x_1,x_2;s),n\right)\quad(n\ge1)\\
		&=(xx_1x_2)^{-n}(x_1x_2)^{-1}\left(\ti_{p_1}(x_1)-t_{n}(p_1;x_1)\right)\left(\ti_{p_2}(x_2)-t_{n}(p_2;x_2)\right)r_3(n),\\
		&\Res\left(g_{p_1,p_2,q_2}(x,x_1,x_2;s),-n\right)\quad(n\ge1)\\
		&=(xx_1x_2)^n\left(\ti_{p_1}(x_1)x_1^{-1}+(-1)^{p_1}t_{n}\Big(p_1;x_1^{-1}\Big)\right)\left(\ti_{p_2}(x_2)x_2^{-1}+(-1)^{p_2}t_{n}\Big(p_2;x_2^{-1}\Big)\right)r_3(-n),\\
		&\Res\left(g_{p_1,p_2,q_2}(x,x_1,x_2;s),-n-1/2\right)\quad(n\ge0)\\
		&=\sum_{k=0}^{p_1+p_2-1}\left((-1)^k\ti_{k+1}(x)-x\ti_{k+1}\Big(x^{-1}\Big)\right)\frac{(xx_1x_2)^nr_3^{(p_1+p_2-k-1)}(-n-1/2)}{(p_1+p_2-k-1)!}\nonumber\\
		&\quad+\sum_{0\le k_1+k_2\le p_2-1,\atop k_1,k_2\ge0}\left((-1)^{k_1}\ti_{k_1+1}(x)-x\ti_{k_1+1}\Big(x^{-1}\Big)\right)\binom{k_2+p_1-1}{p_1-1}\nonumber\\
		&\quad\quad\times\left((-1)^{k_2}\Li_{k_2+p_1}(x_1)+(-1)^{p_1}\zeta_n\Big(k_2+p_1;x_1^{-1}\Big)\right)\frac{(xx_1x_2)^nr_3^{(p_2-k_1-k_2-1)}(-n-1/2)}{(p_2-k_1-k_2-1)!}\nonumber\\
		&\quad+\sum_{0\le k_1+k_2\le p_1-1,\atop k_1,k_2\ge0}\left((-1)^{k_1}\ti_{k_1+1}(x)-x\ti_{k_1+1}\Big(x^{-1}\Big)\right)\binom{k_2+p_2-1}{p_2-1}\nonumber\\
		&\quad\quad\times\left((-1)^{k_2}\Li_{k_2+p_2}(x_2)+(-1)^{p_2}\zeta_n\Big(k_2+p_2;x_2^{-1}\Big)\right)\frac{(xx_1x_2)^nr_3^{(p_1-k_1-k_2-1)}(-n-1/2)}{(p_1-k_1-k_2-1)!},\\
		&\Res\left(g_{p_1,p_2,q_2}(x,x_1,x_2;s),0\right)\\
		&=\sum_{0\le k_1+k_2\le q_2,\atop k_1,k_2\ge0}(x_1x_2)^{-1}\binom{k_1+p_1-1}{p_1-1}\binom{k_2+p_2-1}{p_2-1}\ti_{k_1+p_1}(x_1)\ti_{k_2+p_2}(x_2)\frac{(-1)^{k_1+k_2}R_3^{(q_2-k_1-k_2)}(0)}{(q_2-k_1-k_2)!}\nonumber\\
		&\quad+\sum_{0\le k_1+k_2+k_3\le q_2-1,\atop k_1,k_2,k_3\ge0}\left((-1)^{k_1}\Li_{k_1+1}(x)-\Li_{k_1+1}\Big(x^{-1}\Big)\right)\nonumber\\
		&\quad\quad\times\binom{k_2+p_1-1}{p_1-1}\binom{k_3+p_2-1}{p_2-1}(-1)^{k_2+k_3}\ti_{k_2+p_1}(x_1)\ti_{k_3+p_2}(x_2)\frac{(x_1x_2)^{-1}R_3^{(q_2-k_1-k_2-k_3-1)}(0)}{(q_2-k_1-k_2-k_3-1)!}.
	\end{align*}
	From Lemma \ref{lem-redisue-thm}, we know that
	\begin{align*}
		&\Res\left(g_{p_1,p_2,q_2}(x,x_1,x_2;s),0\right)+\sum_{n=1}^{\infty}\Res\left(g_{p_1,p_2,q_2}(x,x_1,x_2;s),n\right)\\&+\sum_{n=0}^{\infty}\Res\left(g_{p_1,p_2,q_2}(x,x_1,x_2;s),-n-1/2\right)+\sum_{n=1}^{\infty}\Res\left(g_{p_1,p_2,q_2}(x,x_1,x_2;s),-n\right)\\&+\sum_{\alpha\in S_3\backslash\left\{0\right\}}\Res\left(g_{p_1,p_2,q_2}(x,x_1,x_2;s),\alpha\right)=0.
	\end{align*}
	Finally, combining these contributions yields the result.
\end{proof}

Clearly, the results presented in Sections 3 and 4 can be derived in this section by assigning specific values to the rational functions discussed here. Similarly, assigning other values to these rational functions may yield additional results on cyclotomic Euler-type sums, which we leave to interested readers.

\medskip

{\bf Declaration of competing interest.}
The author declares that he has no known competing financial interests or personal relationships that could have
appeared to influence the work reported in this paper.

{\bf Data availability.}
No data was used for the research described in the article.

{\bf Acknowledgments.} Ce Xu is supported by the General Program of Natural Science Foundation of Anhui Province (Grant No. 2508085MA014). Ce Xu gratefully acknowledges the invitation from Professor Chengming Bai of Nankai University to the Chern Institute of Mathematics and from Professor Shaoyun Yi of Xiamen University to the Tianyuan Mathematical Center in Southeast China (TMSE). This work commenced during these visits. Ce Xu also extend his sincere gratitude to Dr. Hongyuan Rui from Sichuan University for his valuable suggestions, which have significantly improved the quality of this paper.

\end{document}